\documentclass[11pt,reqno]{amsart}

\usepackage[table]{xcolor} 
\usepackage[utf8]{inputenc}
\usepackage[T1]{fontenc}
\usepackage{amsmath,amsfonts}
\usepackage{algorithmic}
\usepackage{algorithm}
\usepackage{array}
\usepackage{textcomp}
\usepackage{stfloats}
\usepackage{url}
\usepackage{stmaryrd}
\usepackage{graphicx}
\usepackage{cite}
\usepackage{mathtools}
\usepackage{mathrsfs}
\usepackage{amsthm}
\usepackage{amssymb}
\usepackage[colorlinks=true]{hyperref}
\usepackage{microtype}
\usepackage{subcaption}  
\usepackage{textcomp}
\hypersetup{urlcolor=blue, citecolor=red}

\usepackage{fullpage}
\usepackage{pgfplots}
\usepackage[siunitx]{circuitikz}
\pgfplotsset{compat=1.15}
\usepackage{mathrsfs}
\usetikzlibrary{arrows}
\usetikzlibrary{arrows.meta}
\usetikzlibrary{positioning,fit}
\usetikzlibrary{pgfplots.groupplots} 
\usetikzlibrary{external}
\tikzset{external/optimize=true}

\DisableLigatures{encoding = *, family = * }

\usepackage{pgf}
\definecolor{ballblue}{rgb}{0.13, 0.67, 0.8}
\definecolor{darkpastelgreen}{rgb}{0.01, 0.75, 0.24}
\definecolor{darkorange}{rgb}{1.0, 0.55, 0.0}
\definecolor{brightlavender}{rgb}{0.75, 0.58, 0.89}
\definecolor{cerisepink}{rgb}{0.93, 0.23, 0.51}
\definecolor{medioumgray}{rgb}{0.82, 0.82, 0.82}
\definecolor{lightgray}{rgb}{0.95, 0.95, 0.95}

\usepackage{hyperref}
\usepackage{cleveref}
\usepackage{mathtools}
\hypersetup{
    colorlinks,
    linkcolor={red!80!black},
    citecolor={green!80!black}
} 
\usepackage{amsopn}

\begingroup 
\theoremstyle{plain}
\newtheorem{theorem}{Theorem}[section]

\newtheorem{proposition}[theorem]{Proposition}

\theoremstyle{definition}

\newtheorem{remark}[theorem]{Remark}

\endgroup

\numberwithin{table}{section}
\numberwithin{equation}{section}
\numberwithin{figure}{section}

\def \*#1{\boldsymbol{#1}}

\newcommand{\NN}{\mathbb{N}}

\newcommand{\RR}{\mathbb{R}}

\newcommand{\SNN}{\mathrm{SNN}}
 
\newcommand{\MC}[1]{\mathcal{#1}}
\newcommand{\MB}[1]{\mathbb{#1}}
\newcommand{\inter}[1]{\llbracket #1 \rrbracket}

\newcommand{\lin}{\mathcal{L}^{\,\text{in}}}
\newcommand{\lhid}[1]{\mathcal{L}^{\,\text{hidd}}_{#1}}
\newcommand{\lout}{\mathcal{L}^{\,\text{out}}}
\newcommand{\lread}{\mathcal{L}^{\,\text{read}}}

\newcommand{\nin}{\mathcal{N}^{\,\text{in}}}
\newcommand{\nhid}[2]{\mathcal{N}^{\,\text{hidd}}_{#1,#2}}

\newcommand{\nout}[1]{\mathcal{N}^{\,\text{out}}_#1}
\newcommand{\nread}{\mathcal{N}^{\,\text{read}}}

\newcommand{\loss}{\text{loss}}

\title[]{Spiking Neural Networks: a theoretical framework for Universal Approximation and training}
\author{Umberto Biccari\textsuperscript{\,$\ast$}}  
\address{\textsuperscript{$\ast$}\, Chair of Computational Mathematics, DeustoTech, University of Deusto, Avenida de las Universidades 24, 48007 Bilbao, Basque Country, Spain} 
\email{umberto.biccari@deusto.es}
\thanks{This project has received funding from the European Research Council (ERC) under the European Union's Horizon 2030 research and innovation programme (grant agreement NO: 101096251-CoDeFeL) and from MINECO, Spanish Government, through the Grants TED2021-131390B-I00/ AEI/10.13039/501100011033 DasEl and PID2023-146872OB-I00-DyCMaMod} 

\begin{document}

\begin{abstract}
Spiking Neural Networks (SNNs) are widely regarded as a biologically inspired and energy‑efficient alternative to classical artificial neural networks. Yet, their theoretical foundations remain only partially understood. In this work, we develop a rigorous mathematical analysis of a representative SNN architecture based on Leaky Integrate‑and‑Fire (LIF) neurons with threshold‑reset dynamics. Our contributions are twofold. First, we establish a universal approximation theorem showing that SNNs can approximate continuous functions on compact domains to arbitrary accuracy. The proof relies on a constructive encoding of target values via spike timing and a careful interplay between idealized $\delta$‑driven dynamics and smooth Gaussian‑regularized models. Second, we analyze the qualitative behavior of spike times across layers, proving well‑posedness of the hybrid dynamics and deriving conditions under which spike counts remain stable, decrease, or in exceptional cases increase due to resonance phenomena or overlapping inputs. Together, these results provide a principled foundation for understanding both the expressive power and the dynamical constraints of SNNs, offering theoretical guarantees for their use in classification and signal processing tasks.
\end{abstract}

\maketitle

\section{Introduction}\label{sec:intro}

Spiking Neural Networks (SNNs) have emerged as a promising paradigm for developing energy efficient and biologically plausible models of computation. Unlike classical Artificial Neural Networks (ANNs), which process continuous-valued activations, SNNs transmit information through discrete spikes, closely emulating the signaling mechanisms of real neurons. This event-driven nature makes them particularly appealing for applications in neuromorphic hardware and real-time sensory processing, where sparse activity and low energy consumption are crucial \cite{maass1997networks,eshraghian2023training,gerstner2002spiking,pfeiffer2018deep}. At the same time, SNNs offer a conceptual bridge between neuroscience and machine learning, providing models that are both biologically plausible and computationally powerful.

Despite their conceptual appeal, training SNNs remains challenging due to the non-differentiable nature of the spiking mechanism and the inherent hybrid dynamics governing neuronal behavior. Recent advances \cite{eshraghian2023training} have sought to bridge this gap by importing ideas from deep learning and control theory to derive principled training algorithms that respect the underlying neurodynamics. Yet, the theoretical foundations of SNNs remain less developed than those of ANNs. Two fundamental questions, central to understanding their capabilities, are still only partially resolved. 

The first concerns expressive power: to what extent can SNNs approximate arbitrary nonlinear functions? Early milestones \cite{maass1996lower,maass1997networks} demonstrated that spiking networks are computationally as powerful as Turing machines, and subsequent works explored their approximation capabilities in both rate-based and temporal coding frameworks \cite{funahashi1993approximation,perez2021sparse,tavanaei2019deep}. More recently, surveys such as \cite{bouvier2019spiking,nunes2022spiking} and methodological advances like \cite{eshraghian2023training} have reinforced the role of SNNs within modern deep learning. Yet, most of these results rely on heuristic or asymptotic arguments, leaving a gap for constructive, mathematically rigorous proofs of universal approximation in realistic spiking models.

The second question relates to spike dynamics: how do the hybrid dynamical laws governing spiking neurons constrain or facilitate information transmission across layers? Although the dynamics of Leaky Integrate-and-Fire (LIF) and Hodgkin–Huxley neurons are well understood in computational neuroscience \cite{dayan2005theoretical,hodgkin1952quantitative}, there is still limited mathematical analysis of spike counts, their stability across layers, and the exceptional regimes \textendash\, such as resonance phenomena or overlapping inputs \textendash\, where they may increase. A deeper understanding of these mechanisms is crucial, since the expressive power of SNNs cannot be decoupled from the dynamical rules that govern their activity.

In this work, we contribute to filling both gaps. Our first main result is a constructive Universal Approximation Theorem for SNNs. We prove that SNNs built from LIF neurons with threshold-reset dynamics can approximate any continuous function on compact domains with arbitrary accuracy. The proof explicitly encodes target function values via spike timing, relying on a careful interplay between an idealized $\delta$-driven model\footnote{Here, $\delta$-driven means that the model's formulation involves Dirac $\delta$ distribution.} \textendash\, used for constructive encoding \textendash\, and a Gaussian-regularized model, which ensures well-posedness and differentiability. This dual modeling strategy is essential: the $\delta$-formulation provides the conceptual clarity needed for constructive analysis, while the regularized version ensures well-posed and practically implementable dynamics.

Our second contribution is a qualitative analysis of spike-time propagation across layers. We establish the well-posedness of the hybrid dynamical systems defining membrane potentials, and we identify conditions under which spike counts remain stable, decrease, or exceptionally increase. In particular, we show that under natural assumptions, spike counts form a non-increasing sequence across layers, with precise criteria describing the rare cases \textendash\, such as overlapping presynaptic spikes \textendash\, that may lead to increases. This analysis provides a rigorous mathematical explanation of how SNNs' expressive power interacts with their dynamical constraints.

These two perspectives complement one another: the universal approximation theorem highlights the expressive capacity of SNNs, while the qualitative study of spike times elucidates the dynamical mechanisms through which such capacity is realized or limited. Together, they provide a rigorous mathematical foundation for understanding the power and constraints of SNNs and deliver a principled theoretical foundation for their study. By combining a constructive universal approximation framework with a detailed dynamical analysis, we bridge the gap between the expressive potential of spiking architectures and the constraints imposed by their hybrid dynamics. In doing so, we complement recent empirical and algorithmic advances, offering rigorous guarantees that support the use of SNNs in classification, signal processing, and beyond.

This paper is organized as follows. In Section \ref{sec:framework}, we introduce the general architecture of the SNN under consideration, describing in detail the input, hidden, output, and readout layers together with their governing dynamics. Section \ref{sec:well_posed} is devoted to the well-posedness analysis of these hybrid dynamical systems and to the characterization of spike times across layers, providing conditions for stability and growth of spike counts. In Section \ref{sec:UA}, we present and prove our main result: a constructive Universal Approximation Theorem for SNNs, highlighting the interplay between $\delta$-driven and Gaussian-regularized dynamics. Section \ref{sec:training} develops a principled training framework based on surrogate gradients and optimal control methods, which enables gradient-based optimization of the proposed architecture. Section \ref{sec:experiments} reports simulation experiments that validate the theoretical findings and illustrate the practical behavior of the proposed framework. Finally, Section \ref{sec:conclusions} summarizes our contributions and outlines open questions for future research.

\section{General framework}\label{sec:framework}

As anticipated, SNNs process information in a fundamentally different way from classical ANNs. Rather than working with continuous activations, they rely on discrete events (spikes) that occur in time. These spikes serve as the basic units of communication between neurons, giving rise to dynamics that are both time-dependent and event-driven. This temporal nature makes SNNs closer to biological computation and motivates the mathematical framework we now introduce.

Consider a dataset $\MC{D}\coloneqq\{(\*x_i,y_i)\}_{i=1}^N$ with $\*x_i\in\RR^d$ and $y_i\in\{0,1\}$ for all $i\in\inter{N}\coloneqq \{1,\ldots,N\}$. We want to classify $\MC{D}$ employing SNNs. In particular, let us focus on the following architecture (see Figure \ref{fig:architecture}):
\begin{itemize}
	\item[1.] One input layer $\lin$ with one neuron $\nin$ whose behavior is described by a \textbf{membrane potential}  
	\begin{align*}
		v[\*x_i](t): \RR \to \RR, \quad \text{for all } t\in(0,T),
	\end{align*}
	receiving as input the data $\{\*x_i\}_{i=1}^N$. 
    \item[2.] $L\geq 1$ hidden layers $\{\lhid \ell\}_{\ell=1}^L$, each one of them with $P\geq 1$ neurons $\{\nhid \ell p\}_{p=1}^P$ with membrane potential $\xi_{\ell,p}$. Each of these hidden layers receives as input the output of the layer immediately previous to it, that is, 
	\begin{displaymath}
        \begin{array}{ll} 
		      \*\xi_1 = \*\xi_1[v](t): \RR \to \RR, & \quad \text{for all } t\in(0,T),
            \\[5pt]
            \*\xi_{\ell+1} = \*\xi_{\ell+1}[\*\xi_\ell](t): \RR \to \RR, & \quad \text{for all } t\in(0,T) \text{ and } \ell\in\inter{L-1},
        \end{array}
	\end{displaymath}
    where for all $\ell\in\inter{L}$, $\*\xi_\ell = (\xi_{\ell,1},\ldots,\xi_{\ell,P})\in\RR^P$ denotes the vector collecting the $P$ membrane potentials of $\lhid \ell$. 
	\item[3.] One output layer $\lout$ composed of $P$ neurons $\{\nout p\}_{p=1}^P$ with membrane potentials 
	\begin{align*}
		u_p[\xi_{L,p}](t): \RR\to \RR, \quad \text{for all } t\in(0,T) \text{ and } p\in\inter{P},
	\end{align*}
    each one of them receiving information of the $p$-th neuron in the last hidden layer.
    \item[4.] A final readout layer $\lread$ with one readout neuron $\nread$ aggregating the outputs of $\{\nout p\}_{p=1}^P$ and performing classification.
\end{itemize}

\begin{figure}[!ht]
     \centering
     \includegraphics[scale=0.95]{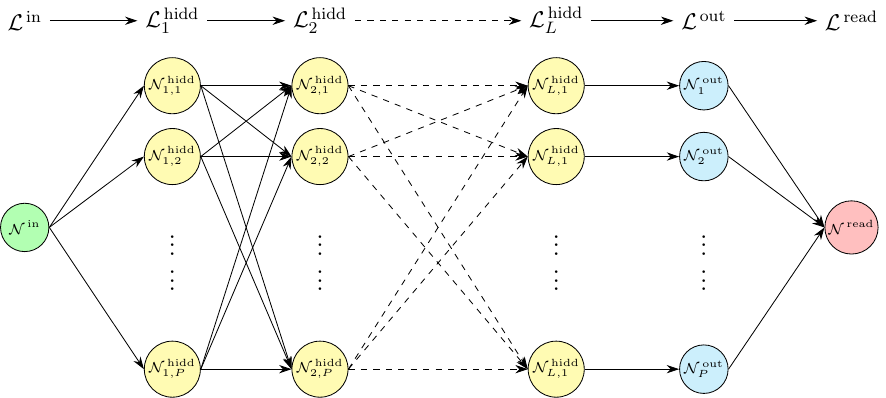}
     \caption{Scheme of the employed SNN architecture. The network is organized into an input layer, several hidden layers, and an output layer. Information is transmitted forward through discrete spike events, and the final layer provides the features used for prediction.}\label{fig:architecture}
\end{figure}

\subsection{Modeling of the membrane potentials}

For the membrane potentials we will employ the \textit{Leaky Integrate}-\textit{and}-\textit{Fire} (LIF) model, which consists in modeling each neuron as a low-pass filter circuit with one resistor and one capacitor. This modeling is applied uniformly across the network: input, hidden, and output layers are all described by LIF-type dynamics. 

From a biological perspective, the LIF model distills key mechanisms of neuronal behavior into a mathematically tractable form. In real neurons, the membrane potential arises from the flow of ions across the cell membrane through specialized channels, while incoming signals arrive as synaptic currents that transiently alter this potential \cite{hodgkin1952quantitative,lapicque1907recherches}. When the accumulated depolarization crosses a critical threshold, the neuron emits an action potential (a spike), after which ion pumps and channel dynamics restore the potential to its resting state \cite{llinas2002thalamocortical}. The LIF model abstracts this process by representing the membrane as an electrical RC circuit: the capacitor mimics the ability of the membrane to store charge, the resistor accounts for the natural leakage of ions through open channels, and synaptic inputs are modeled as external currents charging the circuit. Spiking then occurs when the voltage across this circuit exceeds a prescribed threshold, triggering an instantaneous reset that mirrors the refractory recovery of biological neurons \cite{dayan2005theoretical,gerstner2002spiking,gerstner2014neuronal,koch2004biophysics}.

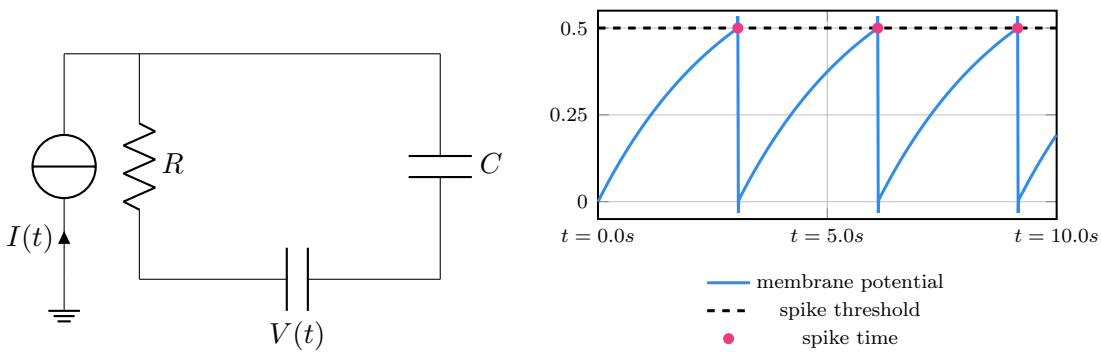
\begin{figure}[!ht]
     \begin{subfigure}[t]{0.45\textwidth}
         \centering
         \begin{circuitikz}[american voltages]
         	
         	\coordinate (G0) at (0,0);
         	\coordinate (N)  at (0,3);
         	\coordinate (Rtop) at (1,3);
         	\coordinate (Rbot) at (1,0);
         	\coordinate (Ctop) at (5,3);
         	\coordinate (Cbot) at (5,0);
         	\coordinate (Vbot) at (1.2,0);
         	
         	\draw (G0) node[ground] {} to[I, i>^=$I(t)$] (N);
         	
         	\draw (N) -- (Rtop) -- (Ctop);
         	
         	\draw (Rtop) to[R, l=$R$] (Rbot);
         	\draw (Ctop) to[C, l=$C$] (Cbot);
         	\draw (Cbot) to [capacitor, l=$V(t)$] (Vbot) -- (Rbot);   
         \end{circuitikz}          
     \end{subfigure}
     \begin{subfigure}[t]{0.45\textwidth}
         \centering
         \includegraphics[scale=0.95]{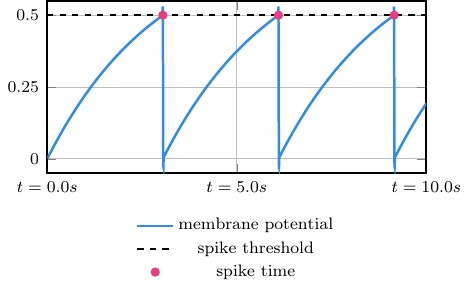}         
     \end{subfigure}
     \caption{General scheme of the LIF model. Left: the RC circuit representation, where the input current $I(t)$ charges the membrane potential $V(t)$. Right: an example of the resulting spiking behavior when $V(t)$ crosses a threshold and is reset. In Sections \ref{subsec:input}, \ref{subsec:hidden} and \ref{subsec:output}, the specific forms of $I(t)$ and $V(t)$ will be detailed for the input, hidden, and output layers of the SNN.}\label{fig:LIF}
\end{figure}

In our architecture, we adopt this LIF abstraction uniformly across all layers, but with a distinction motivated by function: in the input and hidden layers we retain the threshold–reset mechanism to reflect the discrete, event-driven nature of biological spiking, whereas in the output layer we omit the reset so that the potentials evolve continuously and provide stable features for the final static readout. This design preserves the biological inspiration of spike-based communication while ensuring a mathematically well-posed framework suitable for universal approximation.

We now make these modeling choices precise by writing down the dynamical systems that govern the evolution of the membrane potentials in the input, hidden, and output layers, as well as the readout ansatz. The analysis of well-posedness for the LIF models and the characterization of spike times are deferred to Section \ref{sec:well_posed}. 

\subsubsection{Input layer}\label{subsec:input} 

Let $(0,T)$ with $T>0$ be a given time interval, $\theta_v>0$ denote the spike threshold for $\nin$ and set $\tau_v\coloneqq R_vC_v$, with $(R_v,C_v)$ denoting the neuron's resistor and capacitor. For all $i\in\inter{N}$, the presynaptic neuron's potential $v=v[\*x_i]$ evolves according to the hybrid dynamical system 
\begin{align}\label{eq:model_pre}
    \begin{cases}
        \displaystyle\dot{v}(t) = \frac{1}{\tau_v} \Big(-v(t) + \langle \*a,\*x_i\rangle\Big), & t\in (0,T)\setminus\MB{T}_0
        \\
        v(0) = 0 	
        \\
        v_k^- = \theta_v,\quad v_k^+ = 0 
    \end{cases}\tag{\text{In}}
\end{align}
where $\*a\in\RR^d$ is a trainable parameter modulating the input strength, 
\begin{align*}
    \MB{T}_0 = \Big\{t_{0,k}\in (0,T)\,:\, k\in\NN\text{ and } v(t_{0,k})=\theta_v\Big\}
\end{align*}
denotes the sequence of spike times, and 
\begin{align*}
    v_k^\pm \coloneqq \lim_{t\to t_{0,k}^\pm} v(t).
\end{align*}

\subsubsection{Hidden layers}\label{subsec:hidden}

For all $\ell\in\inter{L}$ and $p\in\inter{P}$, let $\theta_{\ell,p}>0$ denote the spike threshold for the neuron $\nhid \ell p$ and set $\tau_{\ell,p}\coloneqq R_{\ell,p} C_{\ell,p}$, with $(R_{\ell,p},C_{\ell,p})$ denoting the neuron's resistor and capacitor. The hidden layers neurons' potential $\xi_{\ell,p}$ evolves according to the hybrid dynamical system 
\begin{align}\label{eq:model_hidd}
    \begin{cases}
        \displaystyle\dot{\xi}_{\ell,p}(t) = \frac{1}{\tau_{\ell,p}} \Big(-\xi_{\ell,p}(t) + \omega_{\ell,p} J_\ell(t)\Big), & \displaystyle t\in (0,T)\setminus\MC{T}_{\ell,p}
        \\
        \xi_{\ell,p}(0) = 0 	
        \\
        \xi_{\ell,p,k}^- = \theta_{\ell,p},\quad \xi_{\ell,p,k}^+ = 0 
    \end{cases}\tag{\text{Hidd}} 
\end{align}
where $\*\omega_\ell=(\omega_{\ell,1},\ldots\omega_{\ell,P})\subset\RR^P$ are trainable parameters, 
\begin{align*}
    \MC{T}_{\ell,p} = \Big\{t_{\ell,p,k}\in (0,T)\,:\, k\in\NN\text{ and } \xi_{\ell,p}(t_{\ell,p,k})=\theta_{\ell,p}\Big\}
\end{align*}
denotes the sequence of spike times, 
\begin{align*}
    \xi_{\ell,p,k}^\pm \coloneqq \lim_{t\to t_{\ell,p,k}^\pm} \xi_{\ell,p}(t),
\end{align*}
and where we have defined the input current
\begin{align}\label{eq:input_J}
    J_\ell(t) \coloneqq \sum_{t^\ast\in\MB{T}_{\ell-1}}\MC{G}_{\mu,t^\ast}(t), \quad\text{ for all } \ell\in\inter{L},
\end{align}
with  
\begin{align*}
    \MC{G}_{\mu,t^\ast}(t) \coloneqq \frac{1}{\mu\sqrt{2\pi}}e^{-\frac{(t-t^\ast)^2}{2\mu^2}}	
\end{align*}
a Gaussian distribution centered at $t^\ast$ with amplitude $\mu>0$ and where  
\begin{align*}
    \MB{T}_{\ell-1} = \bigcup_{p=1}^P \MC{T}_{\ell-1,p}
\end{align*}
collects the spike times generated by all neurons in the $(\ell-1)$-th hidden layer (for $\ell>1$) or, in the case $\ell=1$, by the input neuron.

\begin{remark}\label{rem:mollification}
In a typical SNN architecture, information among the neurons/layers is transmitted through \textbf{spikes}. To describe this phenomenon, one should use the following hybrid ODEs 
\begin{align}\label{eq:model_hidd_deltas}
    \begin{cases}
        \displaystyle\dot{\xi}_{\ell,p}(t) = \frac{1}{\tau_{\ell,p}} \left(-\xi_{\ell,p}(t) + \omega_{\ell,p} \sum_{t^\ast\in\MB{T}_{\ell-1}} \delta(t-t^\ast)\right), & \displaystyle t\in (0,T)\setminus\MC{T}_{\ell,p}
        \\
        \xi_{\ell,p}(0) = 0 	
        \\
        \xi_{\ell,p,k}^- = \theta_{\ell,p},\quad \xi_{\ell,p,k}^+ = 0 
    \end{cases},		
\end{align}
which models more faithfully the impulsive nature of spike transmission. However, due to the presence of the Dirac deltas, \eqref{eq:model_hidd_deltas} is not well-posed in standard ODE theory: the right-hand side is only defined in the sense of distributions, and solutions are not classical trajectories \cite{lakshmikantham1989theory}.

For this reason, for the rigorous analysis in this work, we have decided to focus on models like \eqref{eq:model_hidd} where each delta $\delta(t-t^*)$ is replaced by a smooth Gaussian kernel $G_{\mu,t^*}(t)$. This mollification process is of course motivated by the weak-$\ast$ convergence of $\MC{G}_{\mu,t^\ast}(t)$ to $\delta(t - t^\ast)$: for any test function $\varphi \in C_0^\infty(\RR)$, we have
\begin{align*}
    \lim_{\mu \to 0^+} \int_\RR G_{\mu,t^\ast}(t) \varphi(t) \, dt = \varphi(t^\ast) = \int_\RR \delta(t - t^\ast)\varphi(t) \, dt.
\end{align*}

At the same time, notice that one may still write a natural \textit{formal} solution of \eqref{eq:model_hidd_deltas} by interpreting each delta impulse as producing an instantaneous exponential response in the membrane potential. Concretely, if $\{t^*\}$ are the input spike times, one obtains a symbolic expression of the form
\begin{align*}
    \xi_{\ell,p,\delta}(t) = \frac{\omega_{\ell,p}}{\tau_{\ell,p}} \sum_{t^* \in \MB T_{\ell-1}} e^{-\frac{t-t^*}{\tau_{\ell,p}}},    
\end{align*}
interleaved with threshold-triggered resets. This explicit representation, though heuristic, will later be exploited in the proof of Theorem \ref{thm:UA} to construct spike trains that encode prescribed target values.

In summary, throughout the paper, the delta formulation serves as an idealized proxy for intuition and constructive arguments, while the Gaussian model underpins the rigorous analysis and training framework.
\end{remark}

\subsubsection{Output layer}\label{subsec:output}

Also the postsynaptic neurons $\{\nout p\}_{p=1}^P$ will be modeled as an RC circuit, receiving as inputs the spiking information from the last hidden layer's neurons $\{\nhid L p\}_{p=1}^P$. In particular, we will consider the dynamical systems
\begin{align}\label{eq:model_post} 
    \begin{cases}
        \displaystyle\dot u_p(t) = \frac{1}{\tau_u} \Big(-u_p(t) + w\Phi_p(t)\Big) & t\in(0,T)
        \\
        u_p(0)=0	
    \end{cases}, \quad\text{ for all } p\in\inter{P}, \tag{\text{Out}}		
\end{align}
with $\tau_u>0$, input current
\begin{align}\label{eq:input_Phi}
    \Phi_p(t) \coloneqq \sum_{t^\ast\in\MB{T}_{L,p}} \MC{G}_{\mu,t^\ast}(t),
\end{align}
and $w\in\RR$ a trainable parameter common to each postsynaptic neuron. Also in this case, \eqref{eq:model_post} is a smooth approximation of 
\begin{align}\label{eq:model_post_deltas}
    \begin{cases}
        \displaystyle\dot{u}_p(t) = \frac{1}{\tau_u} \left(-u_p(t) + w \sum_{t^\ast\in\MB{T}_{L,p}} \delta(t-t^\ast)\right), & \displaystyle t\in (0,T)
        \\
        u_p(0) = 0 	        
    \end{cases},		
\end{align}
that we introduce to simplify the analysis.

\subsubsection{Readout layer}\label{subsec:readout}

The final stage of the network is a static readout layer, acting on the membrane potentials produced by the output neurons at the final time $T$. Unlike the preceding layers, this component does not evolve dynamically but instead implements a fixed functional of the form
\begin{align*}
    \MC R(\*u(T))\coloneqq \sum_{p=1}^P \nu_p\sigma\big(u_p(T)-\theta_u\big),    
\end{align*}
where $\*u=(u_1,\ldots,u_P)$ denotes the vector of output potentials, $\theta_u>0$ is a given threshold, $\sigma:\RR\to\RR$ is a Lipschitz non-polynomial activation function, and $\*\nu = (\nu_1,\ldots,\nu_P)\in\RR^P$ are trainable parameters.

This readout realizes a Cybenko-type ansatz \cite{leshno1993multilayer}: the continuous features $u_p(T)$ generated by the dynamical system are non-linearly combined with trainable weights to produce the final network output. The choice of this structure is crucial, since it ensures the universal approximation property of the architecture that will be established in Section \ref{sec:UA}.

\subsubsection{Full SNN mapping}\label{subsec:SNN}

Combining the dynamical evolution of the input, hidden, and output layers with the static readout functional, we obtain a mapping from the input space $\RR^d$ to the real line. For an input datum $x\in\RR^d$, let
\begin{align*}
    v(t,\*x)\in\RR, \quad \{\*\xi_\ell(t,\*x)\}_{\ell=1}^L\subset\RR^P, \quad \*u(t,\*x)\in\RR^P
\end{align*}
denote the membrane potentials of the input, hidden, and output layers, respectively, at time $t$, governed by the dynamical systems specified in Subsections \ref{subsec:input}-\ref{subsec:hidden}-\ref{subsec:output}. The network output corresponding to $\*x$ is then defined by applying the readout functional of Subsection \ref{subsec:readout}:
\begin{align}\label{eq:SNN}
    \SNN(\*x)\coloneqq \MC R(\*u(T,\*x)) = \sum_{p=1}^P \nu_p\sigma\big(u_p(T,\*x)-\theta_u\big).    
\end{align}

This mapping $\SNN:\RR^d\to\RR$ will serve as the fundamental ansatz of the model. In Section \ref{sec:UA} we will show, in the spirit of Cybenko’s analysis \cite{cybenko1989approximation,leshno1993multilayer}, that the class of such mappings possesses the universal approximation property.

Finally, we emphasize that the explicit dependence on $\*x$ is shown here for the sake of clarity, in order to make the network mapping precise. In the dynamical systems of Subsections \ref{subsec:input}-\ref{subsec:hidden}-\ref{subsec:output} this dependence was left implicit, and in Section \ref{sec:well_posed} it will again be suppressed in the notation so as to focus on the dynamical behavior of the LIF models themselves.

\section{Well-posedness analysis}\label{sec:well_posed}

We discuss here the well-posedness and some elementary properties of the dynamical systems describing the evolution of the neurons membranes' potentials. Our main interest will be to characterize the spike times at which information is transmitted from one layer to the next one.

\subsection{Input layer}\label{subsec:input_wp}

To set the stage for our well-posedness analysis, we first examine the input layer, where the emergence of spike times can be described explicitly.

\begin{proposition}\label{prop:wp_input}
For all $\*x_i\in\RR^d$, $\*a\in\RR^d$, $0<\tau_v,\theta_v\in\RR$ and $T>0$ there exists a unique solution $v$ of the dynamical system \eqref{eq:model_pre} such that 
\begin{align*}
    v\in C^\infty\big((0,T)\setminus\MB{T}_0\big),
\end{align*}
where $\MB{T}_0$ denotes the sequence of spike times $\{t_{0,k}\}\subset(0,T)$ at which $v(t_{0,k})=\theta_v$. Moreover, $\MB{T}_0$ can be characterized as follows:
\begin{itemize}
    \item[1.] if $\langle \*a,\*x_i\rangle\leq\theta_v$, then $\MB{T}_0=\varnothing$;
    \item[2.] if $\langle \*a,\*x_i\rangle>\theta_v$, let us  define 
    \begin{align}\label{eq:beta}
	      \beta(\*a,\*x_i)\coloneqq \tau_v\ln\left(\frac{\langle \*a,\*x_i\rangle}{\langle \*a,\*x_i\rangle-\theta_v}\right)
    \end{align}
    and set $K\coloneqq \lfloor T\beta^{-1}\rfloor$. Then
    \begin{align}\label{eq:spiking_pre}
        \MB{T}_0=\begin{cases}
            \varnothing & \text{ if } K=0
            \\
            \{t_{0,k}\}_{k=1}^K \text{ with } t_{0,k}\coloneqq k\beta(a) & \text{ if } K\geq 1.      
        \end{cases}
    \end{align}
    In particular, the spike times are either the empty set or a growing periodic sequence with period $\beta$.
\end{itemize}
\end{proposition}

\begin{proof}
First of all, notice that the solution to \eqref{eq:model_pre} can be computed explicitly, by superposing the solution of 
\begin{align*}
    \begin{cases}
        \displaystyle \dot v(t) = \frac{1}{\tau_v} \Big(-v(t) + \langle \*a,\*x_i\rangle\Big), & t\in (0,T)
        \\
        v(0) = 0 	
    \end{cases},		
\end{align*}
which is given by  
\begin{align}\label{eq:v_solution}
    v(t) = \langle \*a,\*x_i\rangle \left(1-e^{-\frac{t}{\tau_v}}\right),
\end{align}
with the reset mechanism
\begin{align}\label{eq:reset}	
    v_k^- = \lim_{t\to t_{0,k}^-} v(t) = \theta_v, \quad\quad v_k^+ = \lim_{t\to t_{0,k}^+} v(t) = v_k^- -\theta_v = 0. 
\end{align}

This reset mechanism \eqref{eq:reset} is activated whenever $v$ reaches the threshold $\theta_v$ during the time interval $(0,T)$. This happens or not, depending on the input $\langle \*a,\*x_i\rangle$.

\medskip
\noindent\underline{Case 1: $\langle \*a,\*x_i\rangle\leq 0$.} In this case, the solution \eqref{eq:v_solution} is always constant or decreasing in time. Together with $v(0)=0$, this implies that $v(t)\leq 0$ for all $t\in(0,T)$. Since $\theta_v>0$, we therefore have that $\MB{T}_0=\varnothing$ and $v\in C^\infty(0,T)$.

\medskip 
\noindent\underline{Case 2: $0<\langle \*a,\*x_i\rangle\leq\theta_v$.} In this case, the solution \eqref{eq:v_solution} is increasing in time and fulfills the bound
\begin{align*}
    0\leq v(t) < \langle \*a,\*x_i\rangle, \quad\text{ for all } t\in(0,T).
\end{align*}

This in particular implies that $v(t) < \theta_v$ for all $t\in(0,T)$. Hence, we have again that $\MB{T}_0=\varnothing$ and $v\in C^\infty(0,T)$.

\medskip 
\noindent\underline{Case 3: $\langle \*a,\*x_i\rangle>\theta_v$.} Suppose now that $\langle \*a,\*x_i\rangle>\theta_v$. In this case, $v$ grows from $v(0)=0$ and reaches for the first time the threshold $\theta_v$ at a certain $t_{0,1}>0$, which can be computed explicitly by imposing in \eqref{eq:v_solution} that $v(t_{0,1})=\theta_v$. This gives
\begin{align*}
    t_{0,1} = \tau_v\ln\left(\frac{\langle \*a,\*x_i\rangle}{\langle \*a,\*x_i\rangle-\theta_v}\right) = \beta(\*a,\*x_i).
\end{align*}

Notice that, since $\langle \*a,\*x_i\rangle>\theta_v$, this value $t_{0,1}$ is defined and strictly positive. Now, if $\beta>T$, i.e. if $K=\lfloor T\beta^{-1}\rfloor=0$, then $t_{0,1}\not\in (0,T)$. Hence, once again, we have that $\MB{T}_0=\varnothing$ and $v\in C^\infty(0,T)$. Otherwise, we can split 
\begin{align*}
    (0,T) = (0,t_{0,1})\cup\{t_{0,1}\}\cup (t_{0,1},T)    
\end{align*}
and the solution of \eqref{eq:model_pre} is given by
\begin{equation}\label{eq:v_solution_split1}
    \begin{array}{ll}
        \displaystyle v(t) = \langle \*a,\*x_i\rangle \bigg(1-e^{-\frac{t}{\tau_v}}\bigg), & \text{ if } t\in(0,t_{0,1})
        \\[10pt]
        \displaystyle v(t_{0,1}^-) = \theta_v,\;\; v(t_{0,1}^+) = 0
        \\[10pt]
        \displaystyle v(t) = \langle \*a,\*x_i\rangle \bigg(1-e^{-\frac{t-t_{0,1}}{\tau_v}}\bigg), & \text{ if } t\in(t_{0,1},T)
    \end{array}.
\end{equation}

Over the time interval $(t_{0,1},T)$, this solution \eqref{eq:v_solution_split1} grows again from $v(t_{0,1}^+)=0$ and reaches for the second time the threshold $\theta_v$ at a certain $t_{0,2}>t_{0,1}$, which can be once again computed explicitly by imposing in \eqref{eq:v_solution_split1} that $v(t_{0,2})=\theta_v$. This gives
\begin{align*}
    t_{0,2} = t_{0,1}+\tau_v\ln\left(\frac{\langle \*a,\*x_i\rangle}{\langle \*a,\*x_i\rangle-\theta_v}\right) = t_{0,1}+\beta(\*a,\*x_i) = 2\beta(\*a,\*x_i).
\end{align*}

Clearly, this second spike time $t_{0,2}$ will belong to $\MB{T}_0$ if and only if $2\beta<T$. This reasoning can be repeated inductively so to get that, for a given $k\in\NN^\ast$,  
\begin{align*}
    t_{0,k} = k\beta(\*a,\*x_i),
\end{align*}
as long as $k\beta<T$, that is, as long as $k\leq K=\lfloor T\beta^{-1}\rfloor$. The characterization \eqref{eq:spiking_pre} of $\MB{T}_0$ is therefore complete. Moreover, by construction, we immediately see that 
\begin{align*}
    v\in C^\infty\big((0,T)\setminus\MB{T}_0\big).
\end{align*}
\end{proof}

\subsection{Hidden layers}\label{subsec:hidden_wp} 

We analyze here the well-posedness of the hidden layers' dynamical systems \eqref{eq:model_hidd}, and provide a characterization of the spiking times. 

\begin{proposition}\label{prop:wp_hidden}
For all $\ell\in\inter{L}$, $p\in\inter{P}$, $0<\tau_{\ell,p},\theta_{\ell,p}\in\RR$ and $T>0$ there exists a unique solution $\xi_{\ell,p}$ of the dynamical system \eqref{eq:model_hidd} with $J_\ell$ as in \eqref{eq:input_J} such that 
\begin{align*}
    \xi_{\ell,p}\in C^\infty\big((0,T)\setminus\MC T_{\ell,p}\big),
\end{align*}
where $\MC T_{\ell,p}$ denotes the sequence of spike times $\{t_{\ell,p,k}\}\subset(0,T)$ at which $\xi_{\ell,p}(t_{\ell,p,k})=\theta_{\ell,p}$. Moreover, $\MC T_{\ell,p}$ can be characterized as follows:
\begin{itemize}
    \item[1.] if $\MB T_{\ell-1}=\cup_{p=1}^P \MC T_{\ell-1,p}=\varnothing$, then $\MC T_{\ell,p}=\varnothing$;
    \item[2.] if $\MB T_{\ell-1}\neq\varnothing$, then every spike time $t_{\ell,p,k}\in\MC T_{\ell,p}$ satisfies the necessary condition
    \begin{align}\label{eq:spiking_hidd}
        \omega_\ell J_\ell(t_{\ell,p,k})=\theta_{\ell,p}.
    \end{align}
\end{itemize}
\end{proposition}

\begin{proof}
First of all, if $\MB T_{\ell-1}=\varnothing$, we have from \eqref{eq:input_J} that $J_\ell(t)\equiv 0$. This clearly implies that $\xi_{\ell,p}(t)\equiv 0$.

Let us now assume that $\MB T_{\ell-1}\neq\varnothing$. In the same spirit of the proof of Proposition \ref{prop:wp_input}, notice that the solution to \eqref{eq:model_hidd} can be computed explicitly, by gluing together the solutions of 
\begin{align}\label{eq:model_hidd_split}
    \begin{cases}
        \displaystyle\dot{\xi}_{\ell,p}(t) = \frac{1}{\tau_{\ell,p}} \Big(-\xi_{\ell,p}(t) + \omega_\ell J_\ell(t)\Big), & \displaystyle t\in (t_{\ell,p,k-1},t_{\ell,p,k})
        \\
        \xi_{\ell,p}(t_{\ell,p,k-1}) = 0 	
    \end{cases}
\end{align}
and the reset mechanism
\begin{align*}
    \xi_{\ell,p,k}^- = \lim_{t\to t_{\ell,p,k}^-} \xi_{\ell,p}(t) = \theta_{\ell,p}, \quad\quad \xi_{\ell,p,k}^+ = \lim_{t\to t_{\ell,p,k}^+} \xi_{\ell,p}(t) = 0, 
\end{align*}
where we have set $t_{\ell,p,0}=0$. Thanks to the variation of constants formula, the solution of \eqref{eq:model_hidd_split} is given by 
\begin{align}\label{eq:xi_solution_split}
    \xi_{\ell,p}(t) = \frac{\omega_{\ell,p}}{\tau_{\ell,p}}e^{-\frac{t}{\tau_{\ell,p}}}\int_{t_{\ell,p,k-1}}^t e^{\frac{s}{\tau_{\ell,p}}}J_\ell(s)\,ds.
\end{align}

Moreover, since by construction $J_\ell\in C^\infty(0,T)$, we have that $\xi_{\ell,p}\in C^\infty\big(t_{\ell,p,k-1},t_{\ell,p,k})$. Finally, the spike times can be computed by imposing in \eqref{eq:xi_solution_split} that $\xi_{\ell,p}(t_{\ell,p,k})=\theta_{\ell,p}$, that is
\begin{align*}
    \int_{t_{\ell,p,k-1}}^{t_{\ell,p,k}} e^{\frac{s}{\tau_{\ell,p}}}J_\ell(s)\,ds = \frac{\theta_{\ell,p} \tau_{\ell,p}}{\omega_{\ell,p}} e^{\frac{t_{\ell,p,k}}{\tau_{\ell,p}}}.
\end{align*}
Differentiating this last expression with respect to $t_{\ell,p,k}$ we get the identity 
\begin{align*}
    e^{\frac{t_{\ell,p,k}}{\tau_{\ell,p}}}J_\ell(t_{\ell,p,k}) = \frac{\theta_{\ell,p}}{\omega_{\ell,p}} e^{\frac{t_{\ell,p,k}}{\tau_{\ell,p}}}.
\end{align*}

We then see immediately that, if $\xi_{\ell,p}(t_{\ell,p,k})=\theta_{\ell,p}$, we also have $\omega_\ell J_\ell(t_{\ell,p,k})=\theta_{\ell,p}$. Our proof is therefore completed.
\end{proof}

\begin{remark}
Notice that \eqref{eq:spiking_hidd} provides a set of \textbf{candidate times} for spiking. However, not every solution of \eqref{eq:spiking_hidd} corresponds to an actual spike: a candidate time becomes a true spike time only if the membrane potential $\xi_{\ell,p}(t)$, evolving according to \eqref{eq:model_hidd} and subject to reset, actually reaches the threshold $\theta_{\ell,p}$ at that instant. Thus, \eqref{eq:spiking_hidd} is a necessary (but not sufficient) condition characterizing all spike times, in the sense that each spike time belongs to the solution set of \eqref{eq:spiking_hidd} but the converse inclusion need not hold.
\end{remark}

\subsection{Analysis of the spike times' trend}\label{subsec:spike_times_trend} 

We analyze here the spike times' trend among the hidden layers $\{\lhid \ell\}_{\ell=1}^L$. In particular, we are interested in determining under which conditions the number of spike times increases or decreases from one layer to the successive one.

\begin{proposition}\label{prop:spike_times_far}
Fix $\ell\in\inter{L}$, $p\in\inter{P}$, and let $\xi_{\ell,p}(t)$ be the membrane potential of a hidden-layer neuron governed by the hybrid system \eqref{eq:model_hidd}, with input spike times $\MB T_{\ell-1} = \{t_{\ell-1,k}\}_{k=1}^K\subset (0,T)$ and input current $J_\ell(t)$ given by \eqref{eq:input_J}. Given $\mu>0$, suppose there exists $\Delta> 6\mu$ such that the presynaptic spike times $\MB T_{\ell-1}$ satisfy 
\begin{align}\label{eq:separation}
    t_{\ell-1,k+1} - t_{\ell-1,k} \geq \Delta, \quad\text{ for all } k\in\inter{K-1}.    
\end{align}
Assume further that the parameter $\omega_{\ell,p}$ and the threshold $\theta_{\ell,p}$ are such that 
\begin{align}\label{eq:amplitude}
    \omega_{\ell,p} \MC{G}_{\mu,t_{\ell-1,k}} > \theta_{\ell,p}, \quad\text{ for all } k\in\inter{K}.    
\end{align}
Then, it holds the following
\begin{itemize}
    \item[1.] For each presynaptic spike $t_{\ell-1,k}$ there exists at least one postsynaptic spike $t_{\ell,p,j}\in \MC T_{\ell,p}$ satisfying 
    \begin{align}\label{eq:spike_pattern}
        t_{\ell,p,j} \in (t_{\ell-1,k} - \zeta, t_{\ell-1,k}) \quad\text{ with }\quad \zeta\coloneqq \mu\sqrt{2\log\left(\frac{\omega_\ell}{\mu\sqrt{2\pi}\theta_\ell}\right)}.    
    \end{align}
    \item[2.] Every spike generated by the Gaussian $\MC{G}_{\mu,t_{\ell-1,k}}$ lies in the interval $(t_{\ell-1,k} - \zeta, t_{\ell-1,k}+\zeta)$.    
    \item[3.] For each $k\in\inter{K}$, denote by $N_{\max}(k)$ the maximal number of spikes in $\MC T_{\ell,p}$ generated around the Gaussian $\MC G_{\mu,t_{\ell-1,k}}$. Then
    \begin{align*}
        N_{\max}(k) \leq \left\lfloor 1 + \log_2\left(\frac{\omega_{\ell,p}}{\tau_{\ell,p}\theta_{\ell,p}} \right) \right\rfloor,      
    \end{align*}
    and the total number of spike times in $\MC T_{\ell,p}$ satisfies
    \begin{align}\label{eq:spikes_count}
        K \leq |\MC T_{\ell,p}| \leq K\left\lfloor 1 + \log_2\left(\frac{\omega_{\ell,p}}{\tau_{\ell,p}\theta_{\ell,p}} \right) \right\rfloor.
    \end{align}
    In particular, if $\omega_{\ell,p}<2\tau_{\ell,p}\theta_{\ell,p}$, we have $|\MC T_{\ell,p}| = K$.
\end{itemize}
\end{proposition}

\begin{proof}

First of all, because of the separation condition \eqref{eq:separation}, we can isolate and analyze separately the dynamics around each spike time $t_{\ell-1,k}\in\MB{T}_{\ell-1}$. Indeed, for all $k\in\inter{K}$, we have that 
\begin{align*}
    J_\ell(t) \sim \MC{G}_{\mu,t_{\ell-1,k}}(t),
\end{align*}
since $\MC{G}_{\mu,t_{\ell-1,k}}$ is negligible in the interval $[t_{\ell-1,j} - 3\mu, t_{\ell-1,j} + 3\mu]$ with $j\neq k$. 

Moreover, the amplitude condition \eqref{eq:amplitude} ensures that the equation $\omega_{\ell,p} \MC G_{\mu,t_{\ell-1,k}}(t) = \theta_{\ell,p}$ has two solutions
\begin{align*}
    t_{\ell,p,k}^- = t_{\ell-1,k} - \zeta \quad\text{ and }\quad t_{\ell,p,k}^+ = t_{\ell-1,k} + \zeta,     
\end{align*}
with $\zeta$ as in \eqref{eq:spike_pattern}. By continuity, the ODE solution $\xi_{\ell,p}(t)$ must reach the threshold for some 
\begin{align*}
    t_{\ell,k,1}\in (t_{\ell-1,k} - \zeta, t_{\ell-1,k}).     
\end{align*}

Now, because $t_{\ell,k,1}$ occurs on the rising side of the Gaussian $\MC{G}_{\mu,t_{\ell-1,k}}(t)$, at least half of the Gaussian’s total area remains to the right of $t_{\ell,k,1}$. Hence the remaining effective charge after the first reset is at most
\begin{align*}
    Q_{k,1} \leq \frac 12 \frac{\omega_{\ell,p}}{\tau_{\ell,p}}\int_\RR \MC G_{\mu,t_{\ell-1,k}}(t)\,dt = \frac{\omega_{\ell,p}}{2\tau_{\ell,p}}.
\end{align*}

A second spike requires that this residual charge can raise again $\xi_{\ell,p}$ from $0$ to $\theta_{\ell,p}$. In particular, we need 
\begin{align*}
    \frac{\omega_{\ell,p}}{2\tau_{\ell,p}}\geq\theta_{\ell,p}
\end{align*}
or, equivalently, $\omega_{\ell,p}\geq 2\tau_{\ell,p} \theta_{\ell,p}$.   

Suppose this happens and a second spike occurs at $t_{\ell,k,2} \in (t_{\ell-1,k}, t_{\ell-1,k}+\zeta)$. By then, at most half of the remaining bump area has been consumed. So the residual charge after the second reset satisfies
\begin{align*}
    Q_{k,2} \leq \frac{\omega_{\ell,p}}{4\tau_{\ell,p}}.
\end{align*}

Thus a third spike requires $\omega_{\ell,p}\geq 4\tau_{\ell,p} \theta_{\ell,p}$.  In general, after $n$ spikes, the remaining charge is bounded by
\begin{align*}
    Q_{k,n} \leq \frac{\omega_{\ell,p}}{2^n\tau_{\ell,p}}, 
\end{align*}
and an $(n+1)$-th spike requires $\omega_{\ell,p}\geq 2^n \tau_{\ell,p} \theta_{\ell,p}$. Thus the maximal number $N_\text{max}(k)$ of spikes in $\MC T_{\ell,p}$ around $t_{\ell-1,k}\in\MB{T}_{\ell-1}$ satisfies
\begin{align*}
    \frac{\omega_{\ell,p}}{\tau_{\ell,p} \theta_{\ell,p}}\geq 2^{N_\text{max}(k)},   
\end{align*}
i.e. 
\begin{align*}
    N_\text{max} \leq \left\lfloor 1 + \log_2\left(\frac{\omega_{\ell,p}}{\tau_{\ell,p}\theta_{\ell,p}}\right) \right\rfloor, \quad\text{ for all } k\in\inter{K}.    
\end{align*}

Finally, since the Gaussian bumps are independent (by separation hypothesis \eqref{eq:separation}), we can sum over $k$ the contributions to obtain the final estimate \eqref{eq:spikes_count} for the number of spikes $|\MC T_{\ell,p}|$.
\end{proof}

Figure \ref{fig:spike_times_case1} illustrates Proposition \ref{prop:spike_times_far} in the case $P=1$. We fix $\theta_\ell = 0.2$, choose $\mu = 0.8$, and prescribe presynaptic spike times $\MB T_{\ell-1} = \{20, 35, 50\}$, which satisfy the separation condition \eqref{eq:separation}. In the top panel, we set $\tau_\ell = 5$ and $\omega_\ell = 1.5$, so that $\omega_\ell / (\tau_\ell \theta_\ell) = 1.5 < 2$. In this case, the postsynaptic neuron generates exactly three spikes, located at $\mathbb{T}_\ell = \{19.5, 34.5, 49.5\}$. In the bottom panel, instead, we take $\tau_\ell = 5$ and $\omega_\ell = 3$, yielding $\omega_\ell / (\tau_\ell \theta_\ell) = 3 > 2$. Here, the same presynaptic input produces six postsynaptic spikes, namely $\mathbb{T}_\ell = \{19.6, 20.5, 34.7, 35.5, 49.7, 50.5\}$.

 \begin{figure}[!ht]
     \centering     
     \includegraphics[scale=0.95]{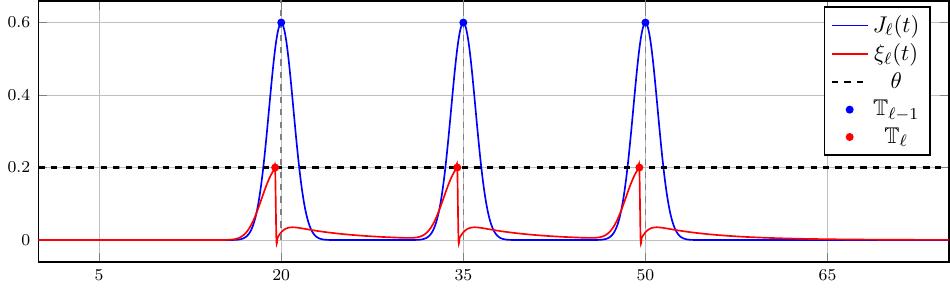}
	 \includegraphics[scale=0.95]{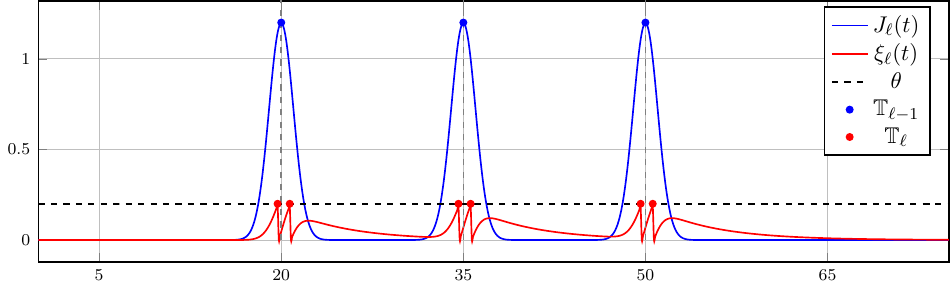}     
     \caption{Graphical illustration of Proposition \ref{prop:spike_times_far}. The plots show how the number of postsynaptic spikes depends on the ratio $\omega_{\ell,p}/(\tau_{\ell,p} \theta_{\ell,p})$. In the top panel, when this ratio is below the critical threshold $2$, each presynaptic spike produces exactly one postsynaptic spike. In the bottom panel, when the ratio exceeds $2$, each presynaptic spike gives rise to two postsynaptic spikes. This demonstrates how the parameters $\tau_{\ell,p}$, $\omega_{\ell,p}$, and $\theta_{\ell,p}$ govern the propagation and multiplication of spike times across layers.}\label{fig:spike_times_case1}
 \end{figure}

\begin{proposition}[Spike behavior under overlapping input bumps]\label{prop:spike_times_close}
Fix $\ell\in\inter{L}$, $p\in\inter{P}$, and let $\xi_{\ell,p}(t)$ be the membrane potential of a hidden-layer neuron governed by the hybrid system \eqref{eq:model_hidd}, with input current $J_\ell(t)$ given by \eqref{eq:input_J} and input spike times $\MB T_{\ell-1} = \{t_{\ell-1,k}\}_{k=1}^K \subset (0,T)$ that may be arbitrarily close. Let
\begin{align*}
    \MC M\coloneqq \Big\{t^*\in(0,T)\,:\, J_\ell'(t^*)=0\text{ and } J_\ell''(t^*)<0\Big\}    
\end{align*}
be the set of strict local maxima of $J_\ell$ and suppose that
\begin{align*}
   \omega_{\ell,p} \max_{t\in\mathcal{M}} J_\ell(t) > \theta_{\ell,p}. 
\end{align*}
Then, it holds the following
\begin{itemize}
    \item[1.] Each maximum $t^*\in\MC M$ generates at least one spike in a neighborhood of $t^*$. In particular, $|\MC T_{\ell,p}|\geq |\MC M|$.
    \item[2.] For each maximum $t^*\in\MC M$, the number of spikes it can generate is bounded by
    \begin{align*}
        N_{\max}(t^*) \leq\left\lfloor 1+\log_2\left(\frac{\omega_{\ell,p}}{\tau_{\ell,p}\theta_{\ell,p}}\right)\right\rfloor.     
    \end{align*}
    Hence
    \begin{align}\label{eq:spikes_count_close}
        |\MC T_{\ell,p}| \leq |\MC M|\left\lfloor 1+\log_2\left(\frac{\omega_{\ell,p}}{\tau_{\ell,p}\theta_{\ell,p}}\right)\right\rfloor.     
    \end{align}
    In particular, if $\omega_{\ell,p}<2 \tau_{\ell,p}\theta_{\ell,p}$, we have $|\MC T_{\ell,p}| = |\MC M|$.
\end{itemize}
\end{proposition}

\begin{proof}

Let $t^*\in\MC M$ be a strict local maximum of $J_\ell(t)$. Since $\xi_{\ell,p}$ is continuous and initially small, and $\omega_{\ell,p} J_\ell(t^*)>\theta_{\ell,p}$ by assumption, the ODE ensures that in a neighborhood of $t^*$, $\xi_{\ell,p}$ grows above the threshold. Thus there exists a first threshold crossing time $t_{\ell,p,1}$ near $t^*$. Therefore, each $t^*\in\MC M$ yields at least one spike and we have $|\MC T_{\ell,p}| \geq |\MC M|$.

To obtain the upper bound \eqref{eq:spikes_count_close}, we adopt the same argument as in the proof of Proposition \ref{prop:spike_times_far}. At each $t^*\in\MC M$, $J_\ell(t)$ behaves like a wide unimodal bump (a sum of overlapping Gaussians). The total effective charge available from this bump is
\begin{align*}
    Q = \frac{\omega_{\ell,p}}{\tau_{\ell,p}}\int_\RR J_\ell(s)\,ds.    
\end{align*}

Now, each spike consumes at least $\theta_{\ell,p}$ units of charge raising $\xi_{\ell,p}$ from $0$ to the threshold $\theta_{\ell,p}$. Because the bump is unimodal, after each threshold crossing, at least half the bump’s residual area lies ahead of the crossing time. Thus the available charge after $n$ spikes is bounded by $Q2^{-n}$. So to produce an $(n+1)$-th spike, we need
$Q2^{-n} \geq \theta_{\ell,p}$, which is equivalent to $\omega_{\ell,p}\geq 2^n \tau_{\ell,p}\theta_{\ell,p}$. Therefore the maximum number of spikes satisfies \eqref{eq:spikes_count_close}. 
\end{proof}

Propositions \ref{prop:spike_times_far} and \ref{prop:spike_times_close} show that, under natural separation or unimodal bump conditions, the number of spikes per layer typically forms a non-increasing sequence. In particular, when presynaptic spike times are sufficiently separated and the ratio
\begin{align*}
    \Gamma_{\ell,p}\coloneqq \frac{\omega_{\ell,p}}{\tau_{\ell,p} \theta_{\ell,p}} < 2,    
\end{align*}
each input bump produces at most one output spike, leading to stability of spike counts across layers.

This quantity $\Gamma_{\ell,p}$ measures the effective input strength per unit time relative to the firing threshold $\theta_{\ell,p}$. Intuitively, if $\Gamma_{\ell,p} < 2$, each bump is too weak to cause more than one firing. On the other hand, if $\Gamma_{\ell,p} \geq 2$, residual charge after a spike may suffice to trigger additional firings. In this sense, we say the neuron operates in a \textbf{high}-\textbf{gain regime} when $\Gamma_{\ell,p} \gg 1$.

Furthermore, the membrane time constant $\tau_{\ell,p}$ controls how quickly the potential leaks away. Here, we can identify two regimes
\begin{itemize}
    \item[1.] \textbf{long memory} (large $\tau_{\ell,p}$) means the neuron retains input longer, making extra spikes more likely;
    \item[2.] \textbf{short memory} (small $\tau_{\ell,p}$) means quick decay, which usually suppresses multiple firings unless bumps overlap.
\end{itemize}

Under these notions, the exceptions to monotonicity in the number of spikes per neuron and layer occur in two main scenarios:
\begin{itemize}
    \item[1.] \textbf{Overlapping presynaptic spikes.} If the intervals between spike times in $\MB T_{\ell-1}$ violate the separation condition \eqref{eq:separation}, Gaussian bumps overlap to form a composite input $J_\ell(t)$ with several local maxima. As Proposition \ref{prop:spike_times_close} indicates, each strict local maximum may induce at least one spike, so that $|\MC T_{\ell,p}| \geq |\MC M|$, where $|\MC M|$ is the number of maxima of $J_\ell$.
    \item[2.] \textbf{High gain with sufficient memory.} If the ratio $\Gamma_{\ell,p}$ satisfies $\Gamma_{\ell,p} \geq 2$, then, as in Proposition \ref{prop:spike_times_far}, a single input bump may yield multiple spikes. Each reset leaves enough residual input charge for another threshold crossing. The maximal number of spikes per bump grows logarithmically with $\Gamma_{\ell,p}$.
\end{itemize}

When both overlap and high gain occur simultaneously, the increase in spike count can be particularly pronounced, as multiple overlapping bumps may each sustain several spikes after resets. Overall, we have the following possibilities
\begin{displaymath}
    \begin{array}{ll}
        |\MC T_{\ell,p}| = |\MB T_{\ell-1}|, & \text{if inputs are separated and } \Gamma_{\ell,p} < 2, 
        \\[6pt]
        |\MC T_{\ell,p}| \leq |\MB T_{\ell-1}|, & \text{if inputs are separated but } \Gamma_{\ell,p} \geq 2, 
        \\[6pt]
        |\MC T_{\ell,p}| \geq |\MC M|, & \text{if overlapping inputs create $|\MC M|$ local maxima}, 
        \\[6pt]
        |\MC T_{\ell,p}| > |\MB T_{\ell-1}|, & \text{if overlapping inputs and } \Gamma_{\ell,p} \geq 2.
    \end{array}
\end{displaymath}

In summary, while the generic trend is non-increasing spike counts across layers, increases may occur under specific conditions: overlapping inputs, sufficiently large gain-to-threshold ratios, and long enough membrane time constants to permit resets between peaks. These situations are exceptional and often unstable, but they provide a precise mechanism for how the number of spike times among layers may grow.

Finally, recall that at the level of the full layer, recall that the set of spike times is given by $\MB T_\ell = \bigcup_{p=1}^P \MC T_{\ell,p}$. Since Proposition \ref{prop:spike_times_far} characterizes the number of spikes neuron-wise, the total number of spikes in the layer is obtained by summing over all $P$ neurons. In particular, the bounds in \eqref{eq:spikes_count} scale linearly with $P$. This observation highlights the role of width: layers with few neurons naturally limit the overall growth of spike counts, whereas wide layers may amplify the total number of spikes, even if each individual neuron remains in a controlled regime. From a biological perspective, this scaling resonates with the principle of metabolic efficiency, since neural systems often rely on sparse populations of neurons to encode information in an energy-efficient way \cite{laughlin2001energy}.

\begin{remark}

According to what we have observed so far, due to the reset dynamics \eqref{eq:model_hidd}, in general each neuron emits at most one spike per sufficiently separated input bump. Hence, if the presynaptic spikes are non-overlapping and well-behaved, the number of output spikes per neuron cannot exceed the number of input spikes, meaning that the sequence of spikes number is non-increasing under non-pathological conditions.

This means that, if we exclude pathological conditions such as resonance effects, overlapping spikes inducing multiple firings per bump, or unusual parameter regimes (tiny thresholds, long memory time constants), then the number $|\MB T_0|$ of spike points in the input layer allows defining an upper bound on the number of spike points in the deeper layers.

Finally, recall that he input layer has a single neuron, governed by the ODE \eqref{eq:model_pre}, and that, according to Proposition \ref{prop:wp_input}, the set of its spike times $\MB T_0 \subset (0,T)$ is 
\begin{itemize}
    \item Empty, if $\langle \*a, \*x \rangle \leq \theta_v$ or if the time horizon $T$ is not large enough so to allow the neuron generate a spike;
    \item A periodically spaced set $\{k \beta\}_{k=1}^K$, otherwise.
\end{itemize}

Hence, the number of spikes in the input layer is determined entirely by the input $\*x$ and by the tuning of the parameter $\*a$ modulating how the input neuron $\nin$ processes $\*x$. 
\end{remark}

\section{Universal approximation}\label{sec:UA}

In this section we investigate the expressive power of the SNN architecture introduced in Section \ref{sec:framework}. Our goal is to show that the class of mappings $\SNN:\RR^d\to\RR$ defined in \eqref{eq:SNN}, obtained by combining the LIF dynamics in the input, hidden, and output layers with the static Cybenko-type readout, is dense in the space of continuous functions on compact subsets of $\RR^d$. This result provides a rigorous justification for the use of the proposed architecture: despite the hybrid dynamical nature of the underlying models, the inclusion of the readout layer ensures that the network retains the universal approximation property established for classical feedforward neural networks \cite{cybenko1989approximation,leshno1993multilayer}.

\begin{theorem}\label{thm:UA}
Let $\rho$ be a Borel probability measure with compact support $\Omega\Subset\RR^d$, and let $f \in C_0(\Omega)\cap L^2(\rho)$. Let $\sigma: \RR\to\RR$ be a Lipschitz, non-polynomial activation function. Then, for every $\varepsilon > 0$, there exists $L,P\in \mathbb{N}^\ast$, $\mu > 0$, and a Spiking Neural Network $\SNN:\RR^d\to\RR$ with architecture given by \eqref{eq:SNN} such that 
\begin{align*}
    \|\SNN - f \|_{L^2(\rho)} < \varepsilon,
\end{align*}
provided
\begin{align*}
    \mu \leq \frac{\varepsilon}{C},    
\end{align*}
where $C = C(\sigma,T,w,\tau_u,\*\nu) > 0$ depends only on the regularity of $\sigma$, the time interval $T$, the parameters $w$ and $\tau_u$ in \eqref{eq:model_post}, and the readout parameters $\*\nu = \{\nu_p\}_{p=1}^P$.
\end{theorem}

\begin{proof}

Before entering the proof, let us highlight that our SNN architecture is rigorously defined in Subsection \ref{subsec:SNN} via the mollified dynamics \eqref{eq:model_post}, which guarantee well-posedness and trainability. Nevertheless, for the constructive part of the universal approximation argument, it will be convenient to temporarily work with the idealized $\delta$-driven system \eqref{eq:model_post_deltas}. As explained in Remark \ref{rem:mollification}, although this system is not well-posed in the classical sense, one may still assign a natural formal solution by superposing exponential responses to delta inputs. This explicit representation provides a convenient proxy that allows us to design spike trains encoding prescribed target values. Its role is justified by the fact that the mollified dynamics \eqref{eq:model_post} approximate the $\delta$-driven ones as $\mu\to 0^+$, ensuring that the final approximation result holds for the actual SNN architecture.

\medskip
\noindent We now divide the proof into multiple steps.

\medskip
\noindent\underline{Step 1: approximation by Shallow Neural Networks.} By \cite[Theorem 1]{leshno1993multilayer}, since $f\in C_0(\Omega) \cap L^2(\rho)$ and $\sigma$ is non-polynomial, there exists an integer $P \in \NN^\ast$, weights $\{\nu_p\}_{p=1}^P \subset\RR$, and directions $\{\*\alpha_p\}_{p=1}^P \subset\Omega$ such that for all $\*x\in\Omega$ the function
\begin{align*}
    f_P(\*x) = \sum_{p=1}^P \nu_p\sigma(\langle \*\alpha_p, \*x \rangle)
\end{align*}
satisfies
\begin{align*}
    \|f - f_P\|_{L^2(\rho)} < \frac{\varepsilon}{2}.    
\end{align*}

\medskip
\noindent\underline{Step 2: Construction of SNN subnets to approximate basis units.} For each $p\in\inter{P}$ and $K\in\NN^\ast$ fixed, we want to design spike times $\MC T_{L,p} = \{t_{L,p,1}, \dots, t_{L,p,K}\}\subset [0,T]$ such that the corresponding solution $u_{p,\delta}$ of \eqref{eq:model_post_deltas} satisfies
\begin{align}\label{eq:u_match}
    u_{p,\delta}(T) = \theta_u + \langle \*\alpha_p, \*x \rangle,    
\end{align}
with $\{\*\alpha_p\}_{p=1}^P$ as in Step 1. To this end, let us first recall (see Remark \ref{rem:mollification}) that this final-time solution is formally given by 
\begin{align*}
    u_{p,\delta}(T) = \frac{w}{\tau_u} \sum_{k=1}^K e^{-\frac{T - t_{L,p,k}}{\tau_u}}.    
\end{align*}
Now, define the function
\begin{align*}
    S: [0, T]^K \to \mathbb{R}, \quad S(t_{L,p,1}, \ldots, t_{L,p,K})\coloneqq \sum_{k=1}^K e^{-\frac{T - t_{L,p,k}}{\tau_u}},    
\end{align*}
and notice that $S$ is continuous $[0, T]^K$ since each exponential term is continuous in $t_{L,p,k}$, and sums of continuous functions are continuous. Moreover, $S$ is strictly increasing in each coordinate $t_{L,p,k}$ and
\begin{align*}
    &\text{if all } t_{L,p,k} = 0, \text{ then } S = Ke^{-\frac{T}{\tau_u}},
    \\
    &\text{if all } t_{L,p,k} = T, \text{ then } S = K.
\end{align*}
Hence, $S$ is surjective onto its range
\begin{align*}
    S([0, T]^K) = \left[K e^{-\frac{T}{\tau_u}}, K \right].
\end{align*}
Let now 
\begin{align*}
    S^\ast\coloneqq \frac{\tau_u}{w} \Big(\theta_u + \langle \*\alpha_j, \*x \rangle \Big),    
\end{align*}
and assume
\begin{align}\label{eq:S_cond}
    S^\ast\in \left[K e^{-\frac{T}{\tau_u}}, K \right].
\end{align}

Since $S$ is continuous and maps a compact set onto this interval, the intermediate value theorem ensures that there exists a tuple $(t_{L,p,1}, \ldots, t_{L,p,K})\in [0, T]^K$ such that $S(t_{L,p,1}, \dots, t_{L,p,K}) = S^\ast$. In conclusion, we can always choose spike times to fulfill \eqref{eq:u_match}, as long as $K$ is large enough so that \eqref{eq:S_cond} holds.

\medskip
\noindent\underline{Step 3: error estimates.} For all $p\in\inter{P}$, let $\MC T_{L,p}$ be constructed as in Step 2, and let $u_p(T)$, $u_{p,\delta}(T)$ be the corresponding final-time solutions of \eqref{eq:model_post} and \eqref{eq:model_post_deltas}. Moreover, let 
\begin{align*}
    \Delta u^{(p)}\coloneqq u_p(T) - u_{p,\delta}(T) &= \frac{w}{\tau_u} e^{-\frac{T}{\tau_u}}\sum_{t^\ast \in\MC T_{L,p}} \left[\int_0^T e^{\frac{s}{\tau_u}} \MC{G}_{\mu,t^\ast}(s)\, ds - e^{\frac{t^\ast}{\tau_u}} \right]
    \\
    &= \frac{w}{\tau_u} e^{-\frac{T}{\tau_u}}\sum_{t^\ast \in\MC T_{L,p}} \left[\int_0^T \varphi(s)\MC{G}_{\mu,t^\ast}(s)\, ds - \varphi(t^\ast) \right],
\end{align*}
with $\varphi(s)\coloneqq e^{\tau_u^{-1}s}$. Then, since $\MC{G}_{\mu,t^\ast}(\cdot)$ is a Gaussian (in particular, a mollifier), we easily get that
\begin{align*}
    \left|\int_0^T \varphi(s) \MC{G}_{\mu,t^*}(s) ds - \varphi(t^*) \right| \leq \mu\sup_{s \in [0,T]} |\varphi'(s)| = \frac{\mu}{\tau_u}e^{\frac{T}{\tau_u}}.  
\end{align*}
Using this and the definition of $\MC T_{L,p}$, we can then estimate
\begin{align*}
    |\Delta u^{(p)}| \leq \frac{\mu|w|}{\tau_u^2} |\MC T_{L,p}| = \frac{K\mu|w|}{\tau_u^2}.
\end{align*}
Hence, if we define $z_p\coloneqq \sigma(u_p(T) - \theta_u)$, using \eqref{eq:u_match} and the above estimate we get
\begin{align*}
    \Big|z_p - \sigma(\langle \*\alpha_p, \*x \rangle)\Big| = \Big|\sigma(u_p(T) - \theta_u) - \sigma(u_{p,\delta}(T) - \theta_u) \Big| \leq \text{Lip}(\sigma) |\Delta u^{(p)}| \leq \frac{K\mu|w|\text{Lip}(\sigma)}{\tau_u^2},     
\end{align*}
where $\text{Lip}(\sigma)$ denotes the Lipschitz constant of $\sigma$. Using this, we can estimate the total error
\begin{align*}
    \|\SNN - f \|_{L^2(\rho)} &\leq \|\SNN - f_P\|_{L^2(\rho)} + \| f_P - f\|_{L^2(\rho)} 
    \\
    &< \sum_{p=1}^P |\nu_p| \cdot \left\|\sigma(u_p(T) - \theta_u) - \sigma(\langle \*\alpha_p, \*x \rangle)\right\|_{L^2(\rho)} + \frac{\varepsilon}{2} 
    \\
    &\leq \frac{K\mu|w|\text{Lip}(\sigma)}{\tau_u^2}\sum_{p=1}^P |\nu_p| + \frac{\varepsilon}{2} \leq \frac{K\mu|w|\|\*\nu\|_1\text{Lip}(\sigma)}{\tau_u^2} + \frac{\varepsilon}{2}.
\end{align*}
Hence, if 
\begin{align*}
    \mu\leq \frac{\varepsilon\tau_u^2}{2K|w|\|\*\nu\|_1\text{Lip}(\sigma)},
\end{align*}
we finally have $\|\SNN - f \|_{L^2(\rho)} < \varepsilon$.
\end{proof}

\subsection{On the necessity of delta and Gaussian dynamics in the universal approximation argument}\label{subsec:deltas_UA}

The proof of Theorem \ref{thm:UA} relies crucially on the interplay between two different formulations of the spiking neuron dynamics: the idealized model \eqref{eq:model_post_deltas} involving Dirac delta functions, and the regularized model \eqref{eq:model_post} based on smooth Gaussian kernels. Each of these plays an essential but distinct role in establishing the universal approximation property, and neither alone would suffice. 

Concretely, the $\delta$-based model \eqref{eq:model_post_deltas} is employed as a mathematical abstraction to construct spike trains that precisely encode the values of the target function to be approximated. At this regards, we shall highlight that, although spike times are generated by the intrinsic dynamics of the LIF model, our constructive arguments show that, by suitable parameter choices, one can ensure that the resulting dynamics produce spike trains with prescribed timings. 

In this sense, the spike times can be designed indirectly. In our case, because the response of $u_{p,\delta}(T)$ to a spike at time $t^*$ is an exponential term 
\begin{align*}
    e^{-\frac{T - t^*}{\tau_u}},    
\end{align*}
the final value becomes an explicit sum:
\begin{align*}
    u_{p,\delta}(T) = \frac{w}{\tau_u} \sum_{k=1}^K e^{-\frac{T - t_{L,p,k}}{\tau_u}}.    
\end{align*}

This formula allows for direct control of the output by carefully choosing the spike times $\{t_{L,p,k}\}$, since each spike contributes independently and additively. Thus, one can solve an inverse problem \textendash\, finding spike times that realize a prescribed value \textendash\, using a well-understood basis of exponential functions.

This tractability fails in the Gaussian-regularized model, where the dynamics are instead governed by a model \eqref{eq:model_post} involving smooth Gaussian bumps centered at the spike times. In this setting, the solution $u_p(T)$ is no longer a clean sum of isolated terms \textendash\, each spike's contribution spreads over time, and neighboring spikes interfere with one another due to the overlap of Gaussians and the integration over the full interval. As a result, the map from spike times to $u_p(T)$ becomes nonlinear, entangled, and difficult to invert explicitly, especially as the number of spikes grows. Therefore, the Gaussian model does not allow a direct constructive encoding of target values.

On the other hand, the Gaussian-regularized model is essential for defining the actual network implementation, as it yields smooth, well-posed dynamics that are amenable to both mathematical analysis and numerical simulation. In contrast, the $\delta$-based model involves impulsive right-hand sides with Dirac distributions, which make the differential equations ill-posed in the classical sense. Specifically:
\begin{itemize}
    \item[1.] The equation does not admit solutions in standard function spaces like $C^1$ or $L^2$, but only in the sense of distributions or measures.  
    \item[2.] The membrane potential $u_{p,\delta}(t)$ becomes discontinuous or non-differentiable at spike times, which complicates the definition and uniqueness of solutions.
\end{itemize}

By replacing each delta function $\delta(t - t^\ast)$ with a smooth, compactly supported approximation $G_{\mu,t^\ast}(t)$, we recover a system with continuous and differentiable dynamics that admits a unique classical solution and can be treated using standard ODE theory.

In particular, the membrane potential $u_p(t)$ becomes a smooth function of time, even in the presence of multiple spikes, and the dependence of $u_p(T)$ on spike times and network parameters becomes differentiable, which is crucial for computing gradients and training the network via standard optimization techniques.

Therefore, the Gaussian model serves as a practically and mathematically viable surrogate for the ideal $\delta$-based model. While it does not allow for explicit spike-time design, its smoothness enables trainability and implementability, and we recover expressiveness by showing that its behavior converges to that of the delta model as the smoothing parameter $\mu \to 0^+$.

In conclusion, both systems are indispensable: the $\delta$-based model provides a constructive encoding mechanism for function values via spike timing, while the Gaussian-based model ensures smooth realizability and robust approximation of the desired behavior.

This dual modeling strategy mirrors classical approximation theory, where one constructs an exact interpolant in an idealized basis (e.g., splines, wavelets) and then implements an approximate version via practical basis functions \cite{devore1993constructive,haar1911theorie,meyer1992wavelets}. In our setting, spikes and their temporal summation play the role of such basis elements, and Gaussian smoothing ensures continuity and differentiability without sacrificing expressive power.

\subsection{On the role of hidden layers and the depth–width tradeoff}

The proof of Theorem \ref{thm:UA} exploits only the explicit representation of the output potentials $u_p(T)$, without relying on the detailed dynamics of the input and hidden layers. One might therefore wonder whether the hidden layers are dispensable, and whether an architecture consisting solely of the output layer \eqref{eq:model_post} would already suffice for universal approximation.

From a purely functional-analytic standpoint, this is correct: universality follows from the ability of the output neurons to generate sums of exponential responses to spike times, which in principle can be prescribed directly. In this abstract sense, hidden layers do not affect the \textit{existence} of an approximating network. Nevertheless, the hidden layers remain indispensable for several reasons:
\begin{itemize}
    \item[1.] \textbf{Encoding mechanism.} The output layer does not receive the input $\*x$ directly: it only integrates spike trains. Without hidden layers, there is no mechanism to transform the input into spike patterns that the output neurons can process. In other words, the hidden layers are the only part of the architecture that connects the static input domain to the temporal structure required by the output layer. This role cannot be bypassed without changing the model itself.
    \item[2.] \textbf{Information richness and scalability.} The sets of spike trains $\MC T_{\ell,p}$ generated in the hidden layers' neurons form the effective dictionary from which the output layer builds its exponential sums. The number of layers $L$ and neurons per layer $P$ directly determine how large and flexible this dictionary is. A shallow or narrow hidden representation may only generate a limited repertoire of spike trains, leading to poor expressive capacity in practice even if universality holds in principle.
    \item[3.] \textbf{Minimal architectures and analogy with neural ODEs.} In analogy with the classical universal approximation theorems, one hidden layer with sufficiently many neurons is expected to suffice for universality. Conversely, one may drastically reduce the width if depth is increased: with only one neuron per layer but sufficiently many layers, universality should still hold. This mirrors the results of \cite{alvarez2024interplay}, where the authors prove that in ResNet-type architectures interpolation can be achieved either with a single wide layer or with many narrow layers. While our proof does not establish such a precise depth–width tradeoff for SNNs, it is consistent with this expectation and suggests that a similar phenomenon underlies the universality of spiking architectures.
\end{itemize}

In summary, although the formal proof of Theorem \ref{thm:UA} is phrased entirely in terms of the output layer, the hidden layers are structurally necessary to link the static input $\*x$ with the temporal spike-based encoding on which the output relies. Their number and size control the richness of the available encodings, and therefore the efficiency and robustness of approximation. In this sense, the hidden layers are not an optional embellishment, but the essential mechanism through which SNNs realize universal approximation.

To see this more concretely, note that if one were to remove all hidden layers and keep only the output dynamics, then the spike times feeding the output neurons would have to be fixed a priori, say $\MC T=\{t^*_1,\dots,t^*_m\}$, with no dependence on the input $\*x$. The outputs would then be only dependent on the time horizon $T$. Such a device can only approximate constant functions on the input space, and in that sense it is not a neural network at all, since it lacks an input–output relation.

By contrast, with even a single hidden layer the spike times $\MC T_{1,p}(\*x)$ are generated dynamically in response to $\*x$. The output then becomes dependent on $\*x$ through the spike locations. This input dependence is exactly what gives SNNs their expressive power. In this sense, the hidden layers are therefore not cosmetic additions but the indispensable mechanism that transforms the output dynamics into a genuine SNN capable of universal approximation.

\subsection{Extension to multiclass classification} 

Although in our analysis we have restricted ourselves to the case of binary classification, where the dataset labels are of the form $y_i \in \{0,1\}$, the framework extends naturally to an arbitrary number of classes without altering the dynamical analysis or the universal approximation result. This requires a simple modification, which consists in replacing the single readout neuron $\nread$ by a layer of $C \geq 2$ readout neurons $\{\nread_c\}_{c=1}^C$, each producing a final-time potential
\begin{align*}
   \MC R^{(c)}\big(\*u(T,\*x_i)\big), \quad c \in\inter{C}, 
\end{align*}
for each datum $\*x_i$, and generating the predictions via the softmax mapping:
\begin{align*}
   y^{(c)}_i = \MC R^{(c)}\big(\*u(T,\*x_i)\big)\left(\sum_{j=1}^C \MC R^{(j)}(\*u(T,\*x_i))\right)^{-1}, \quad c\in\inter{C}. 
\end{align*}

Notice that the well-posedness results of Section \ref{sec:well_posed}, which establish the existence, uniqueness, and characterization of spike times, remain unchanged. They are proved neuron-wise and are independent of the number of readout neurons.

Moreover, the Universal Approximation Theorem \ref{thm:UA}, stated for scalar-valued target functions, extends immediately to vector-valued mappings $f:\Omega\to\RR^C$ by applying the construction component-wise. 

In conclusion, the transition from binary to multiclass classification requires only a modification of the final readout layer, and all the analytical results on well-posedness and universal approximation remain valid without major alteration.

\section{Learning framework}\label{sec:training}
Denote 
\begin{equation}\label{eq:params_stacked}
    \begin{array}{l}
        \Upsilon\coloneqq \Big(\*a,(\*\omega_1,\ldots,\*\omega_L), w\Big)\in\RR^d\times\RR^{PL}\times\RR
        \\
        \*\nu = (\nu_1,\ldots,\nu_P)\in\RR^P
    \end{array}
\end{equation}
the parameters of our SNN, with $\Upsilon$ indicating the ``dynamics'' parameters of the input, hidden and output layers, and $\nu$ the parameters of the static readout layer. Their training will be carried out through standard empirical risk minimization, i.e. minimizing the cost functional 
\begin{align}\label{eq:functional}
    G(\Upsilon,\*\nu) = \frac 1N\sum_{i=1}^N \loss\Big(\MC R(\*u(T,\*x_i),y_i\Big) + \gamma \Big(\|\Upsilon\|^2 + \|\nu\|^2\Big),
\end{align}
where $\gamma>0$ is a regularization parameter, $\sigma:\RR\to (0,1)$ is a Lipschitz activation function, $\loss:\RR\to\RR^+$ is a suitable loss function, and 
\begin{align*}
    \|\Upsilon\|^2 = \|\*a\|^2 + \sum_{\ell=1}^L \|\*\omega_\ell\|^2 + |w|^2.
\end{align*}

To minimize \eqref{eq:functional} we will apply a gradient-based algorithm, which requires computing the gradients $\nabla_\Upsilon G$ and $\nabla_{\*\nu}G$ with respect to the model parameters. 

For what concerns the readout parameters $\*\nu$, this gradient is obtained as a direct application of the chain rules for derivatives. In particular, if we denote by $z_i = \MC R(\*u(T,\*x_i))$ the network's prediction on the $i$-th datum $\*x_i$, we have that
\begin{align*}
    \nabla_{\nu_p} G = 2\gamma\nu_p + \frac 1N \sum_{i=1}^N \partial_{z_i}\loss(z_i,y_i)\sigma(u_p(T)-\theta_u), \quad\text{ for all } p\in\inter{P}.
\end{align*}

As for $\nabla_\Upsilon G$, the computation is more delicate due to the reset mechanisms in the input and hidden layers, which introduces non-differentiability points and jumps in the state trajectories of the LIF dynamics. To rigorously address this issue, we will apply a ``surrogate gradient'' approach, based on the mollification of the discontinuous spike events by replacing the hard thresholding in the LIF model with differentiable surrogate functions (see, e.g., \cite{eshraghian2023training}). This yields smooth approximations of the state trajectories and well-defined gradients $\nabla_\Upsilon G$ while preserving the qualitative spiking behavior. Once the dynamics are regularized, the computation of $\nabla_\Upsilon G$ can be rigorously formulated through an optimal control framework. The resulting first-order optimality system consists of a coupled forward–backward problem: the forward dynamics describe the evolution of membrane potentials and synaptic currents, while the adjoint dynamics propagate the sensitivity of the objective functional with respect to state perturbations. This adjoint methodology provides a principled way to evaluate parameter gradients, in direct analogy with Pontryagin's Maximum Principle \cite{pontrjagin1962mathematical,sontag2013mathematical,trelat2005controle} or backpropagation through time \cite{lecun2015deep}.

\subsection{Computation of $\nabla_\Upsilon G$} 
Let 
\begin{align*}
    \MB T \coloneqq \bigcup_{\ell=0}^L \MB T_\ell 
\end{align*}
be the set of all spike times of the input and hidden layers. For all $t\in (0,T)$ and parameters $\Upsilon$ as in \eqref{eq:params_stacked}, we define the stacked state vector
\begin{align*}
    X(t) = \Big(v(t),\{\*\xi_\ell(t)\}_{\ell=1}^L,\*u(t)\Big)^\top = \Big(v(t),\big\{(\xi_{\ell,1}(t),\ldots,\xi_{\ell,P}(t))\big\}_{\ell=1}^L,\big(u_1(t),\ldots,u_P(t)\big)\Big)^\top 
\end{align*}
and the stacked dynamics 
\begin{align*}
    \MC F(X(t),\Upsilon) = \Big(f_v, \{f_{\ell,1},\ldots,f_{\ell,P}\}_{p=1}^P, \{f_{u_p}\}_{p=1}^P\Big)^\top,
\end{align*} 
where
\begin{equation}\label{eq:forward}
    \begin{array}{lll}
        \text{Input layer} & f_v\coloneqq \tau_v^{-1}\Big(-v+\langle \*a,\*x_i\rangle\Big) & \text{ for all } i\in\inter{N}
        \\
        \text{Hidden layers} & f_{\xi_{\ell,p}}\coloneqq \tau_{\ell,p}^{-1}\Big(-\xi_{\ell,p}+\omega_{\ell,p}J_\ell\Big) & \text{ for all } \ell\in\inter{L}\text{ and } p\in\inter{P}
        \\
        \text{Output layer} & f_{u_p}\coloneqq \tau_u^{-1}\Big(-u_p+w\Phi_p\Big) & \text{ for all } p\in\inter{P}
    \end{array}
\end{equation}

With this notations, the input-hidden-output layer evolution of our SNN is governed by the ODE system
\begin{align}\label{eq:dynamics_stacked}
    \begin{cases}
        \dot X(t) = \MC F(X(t),\Upsilon), & t\in (0,T)\setminus\MB T
        \\
        X(0)=0
        \\
        X^+ = \MC J(X^-)
    \end{cases}
\end{align}
where $X^+ = \MC J(X^-)$ indicates the jump conditions
\begin{align*}
    v_k^- = \theta_v,\; v_k^+=0 \quad\text{ and }\quad \xi_{\ell,p,k}^- = \theta_{\ell,p},\; \xi_{\ell,p,k}^+=0.
\end{align*}
Moreover, by applying standard variational theory, we know that 
\begin{align*}
    \nabla_\Upsilon G = 2\gamma\Upsilon -\int_0^T \lambda(t)^\top\partial_\Upsilon \MC F\big(X(t),\Upsilon\big)\,dt,
\end{align*}
where the adjoint states
\begin{align*}
    \lambda(t) = \Big(\lambda_v(t),\{\lambda_{\*\xi_\ell}(t)\}_{\ell=1}^L,\lambda_{\*u(t)}\Big) = \Big(\lambda_v(t),\{(\lambda_{\xi_{\ell,1}}(t),\ldots,\lambda_{\xi_{\ell,P}}(t))\}_{\ell=1}^L,(\lambda_{u_1}(t),\ldots,\lambda_{u_P}(t))\Big).
\end{align*}
formally satisfy the backward dynamics
\begin{align}\label{eq:adjoint_formal}
    \begin{cases}
        -\dot\lambda(t) = \Big(\partial_{X(t)} \MC F(X(t),\Upsilon)\Big)^\top \lambda(t), & t\in (0,T)\setminus\MB T
        \\[5pt]
        \displaystyle\lambda(T) = \partial_{X(T)}\left(\frac 1N\sum_{i=1}^N \Psi_i\right)   
        \\
        \text{Adjoint jump condition}
    \end{cases}
\end{align}
where for all $i\in\inter{N}$ we have denoted
\begin{align*}
    \Psi_i = \loss\Big(\MC R(\*u(T,\*x_i),y_i\Big).
\end{align*}

The ``adjoint jump condition'' in \eqref{eq:adjoint_formal} captures how the discontinuous reset of the forward dynamics at spike times affects the backward propagation of sensitivities. Intuitively, whenever the membrane potential is reset in the forward system \eqref{eq:dynamics_stacked}, the adjoint variables must also be updated to account for the sudden change in the state trajectory. This ensures that the contribution of each spike event to the overall objective functional is correctly transmitted backward in time. In other words, the jump condition enforces consistency between the discontinuous forward evolution and the backward sensitivity propagation, guaranteeing that the computed gradient faithfully reflects the effect of each reset event on the final loss. However, explicitly handling these jumps makes the adjoint equations cumbersome and the resulting gradients difficult to use in practice. This motivates the introduction of a mollification strategy to smooth out the discontinuities and eliminate the need for explicit jump conditions.

\medskip 
\noindent Concretely, let $\zeta>0$ and consider the smooth reset mollifier
\begin{align}\label{eq:mollifier}
     H_\zeta:\RR\to(0,1), \quad H_\zeta(s)\coloneqq \frac 12\left(1 + \tanh\left(\frac{\zeta s}{2}\right)\right)    
\end{align}
and the (component-wise) discharge map
\begin{align}\label{eq:discharge}
    D_\zeta(s;\theta)\coloneqq \big(1 - H_\zeta(s-\theta)\big)s.    
\end{align}
Notice that, as $\alpha\to +\infty$, we have the point-wise convergence
\begin{align*}
    H_\zeta(s)\to H(s)= \begin{cases} 1 & \text{ if } s>0 \\ \frac 12 & \text{ if } s=0 \\ 0 & \text{ if } s<0 \end{cases} \quad\text{ and }\quad 
    D_\zeta(s;\theta)\to D(s;\theta)= \begin{cases} s & \text{ if } s>\theta \\ \frac \theta2 & \text{ if } s=\theta \\ 0 & \text{ if } s<\theta \end{cases}
\end{align*}

Then, for the input neuron $v$ with threshold $\theta_v$ and for each hidden neuron $\xi_{\ell,p}$ with threshold $\theta_{\ell,p}$ and current $J_\ell(t)$ as in \eqref{eq:input_J}, consider the mollified dynamics
\begin{align}\label{eq:input_mollified}
    \begin{cases}
        \dot v(t) = \frac{1}{\tau_v}\Big(-v(t)+\langle \*a,\*x_i\rangle\Big), & t\in (0,T)\setminus \MB T_0
        \\
        v(0)=0,
        \\
        v(t^{\ast +}) = \MC D_\zeta \big(v(t^{\ast -});\theta_v\big), & \text{for all }t^*\in \MB T_0
    \end{cases}
\end{align}
and
\begin{align}\label{eq:hidden_mollified}
    \begin{cases}
        \dot\xi_{\ell,p}(t)=\frac{1}{\tau_{\ell,p}}\Big(-\xi_{\ell,p}(t)+\omega_{\ell,p}J_\ell(t)\Big), & t\in (0,T)\setminus \MC T_{\ell,p}
        \\
        \xi_{\ell,p}(0)=0,
        \\
        \xi_{\ell,p}(t^{\ast +}) = \MC D_\zeta \big(\xi_{\ell,p}(t^{\ast -});\theta_{\ell,p}\big), & \text{for all }t^*\in \MC T_{\ell,p}
    \end{cases}
\end{align}

With these mollifications \eqref{eq:input_mollified} and \eqref{eq:hidden_mollified}, the flow induced by $\MC F$ is differentiable in $X$ almost everywhere and admits well-defined Jacobians across events. Consequently, the stacked adjoint is given in the classical sense by
\begin{align*}
    \begin{cases}
        -\dot\lambda(t) = \Big(\partial_{X(t)} \MC F(X(t),\Upsilon)\Big)^\top \lambda(t), & t\in (0,T)
        \\[5pt]
        \displaystyle\lambda(T) = \partial_{X(T)}\left(\frac 1N\sum_{i=1}^N \Psi_i\right)           
    \end{cases}
\end{align*}
and all gradients with respect to $\Upsilon$ can be computed directly by automatic differentiation through the (discretized) forward solver. In the limit $\zeta\to +\infty$, the mollified dynamics recover the original hybrid model (see Section \ref{subsec:mollified_convergence}) while retaining numerically stable gradients for training.

Finally, let us stress that there is no need of introducing a mollified version of the output neurons' dynamics, since in \eqref{eq:model_post} we have introduced no reset mechanism and the solution $u_p$ is smooth over the whole time interval $(0,T)$.

\subsubsection{Convergence of the mollified dynamics}\label{subsec:mollified_convergence}

For completeness, we analyze here the convergence of the mollified dynamics \eqref{eq:input_mollified} and \eqref{eq:hidden_mollified} to their hybrid counterparts \eqref{eq:model_pre} and \eqref{eq:model_hidd}. 

\begin{proposition}\label{prop:relaxation_input}
Let $T>0$ and fix $\*a\in\RR^d$, $\tau_v>0$, and $\theta_v>0$. Let $v$ denote the solution of \eqref{eq:model_pre} with spike times set $\MB T_0=\{t_{0,k}\}_{k=1}^K$ characterized in Proposition \ref{prop:wp_input}. We assume all spikes are transversal threshold crossings, i.e. $\dot v(t_{0,k})\neq 0$ for all $k\in\inter{K}$. Let $\{H_\zeta\}_{\zeta>0}$ be a family of smooth reset mollifiers $H_\zeta:\RR\to(0,1)$ defined as in \eqref{eq:mollifier}. For each $\zeta>0$ define the discharge map $D_\zeta(z;\theta)$ as in \eqref{eq:discharge} and let $v_\zeta$ be the solution to the reset–mollified input dynamics \eqref{eq:input_mollified}. Then 
\begin{align*}
    \lim_{\zeta\to +\infty} \sup_{t\in(0,T)}|v_\zeta(t)-v(t)|=0,    
\end{align*}
that is, $v_\zeta \to v$ uniformly on $(0,T)$.
\end{proposition}

\begin{proof} 

First of all, let us notice that, if $\langle \*a,\*x_i\rangle \leq\theta_v$, then we have from Proposition \ref{prop:wp_input} that $\MB T_0=\varnothing$ and, clearly, $v\equiv v_\zeta$ for all $t\in (0,T)$. 

\smallskip
\noindent Assume now that $\langle \*a,\*x_i\rangle >\theta_v$. We shall proceed through two steps.

\medskip
\noindent\underline{Step 1: agreement of flows between spikes.} 
Let $\MB T_0=\{t_{0,k}\}_{k=1}^K\subset(0,T)$ as in Proposition \ref{prop:wp_input}. On each subinterval $[t_{0,k},t_{0,k+1})$, both $v$ and $v_\zeta$ solve the same linear ODE with (possibly) different initial data at $t_{0,k}^+$. Hence, for $t\in[t_{0,k},t_{0,k+1})$, we have
\begin{align}\label{eq:v_diff}
    v_\zeta(t)-v(t) = \big(v_\zeta(t_{0,k}^+)-v(t_{0,k}^+)\big)e^{-\frac{t-t_{0,k}}{\tau_v}}.    
\end{align}
Now, at each spike time $t_{0,k}$, we have
\begin{align*}
    v(t_{0,k}^+)=0 \quad\text{ and }\quad v_\zeta(t_{0,k}^+) = \big(1-H_\zeta(v_\zeta(t_{0,k}^-)-\theta_v)\big)v_\zeta(t_{0,k}^-).    
\end{align*}

Moreover, we also know that $v_\zeta(t_{0,k}^-)\to v(t_{0,k}^-)=\theta_v$ and, because the threshold crossing in the hybrid system is transversal ($\dot v(t_{0,k})\neq 0)$, for sufficiently large $\zeta$, the mollified trajectory approaches the threshold from the super-threshold side: $v_\zeta(t_{0,k})-\theta_v > 0$. Consequently,
\begin{align*}
    H_\zeta\big(v_\zeta(t_{0,k}^+)-\theta_v\big)\to 1 \quad\text{ as }\zeta\to +\infty
\end{align*}
and
\begin{align*}
    \lim_{\zeta\to +\infty} v_\zeta(t_{0,k}^+) = \lim_{\zeta\to +\infty} \big(1-H_\zeta(\cdot)\big)v_\zeta(t_{0,k}^-) = 0 = v(t_{0,k}^+).
\end{align*}
Therefore, 
\begin{align}\label{eq:v_diff2}
    e_{\zeta,k}\coloneqq |v_\zeta(t_{0,k}^+)-v(t_{0,k}^+)| \to 0 \text{ as } \zeta\to +\infty \text{ for every } k\in\inter{K}.
\end{align}
Combining \eqref{eq:v_diff} and \eqref{eq:v_diff2} yields, for all $t\in[t_{0,k},t_{0,k+1})$,
\begin{align*}
    \sup_{t\in[t_{0,k},t_{0,k+1})}|v_\zeta(t)-v(t)| \leq e_{\zeta,k}\,e^{-\frac{t-t_{0,k}}{\tau_v}} \leq e_{\zeta,k}.    
\end{align*}
Taking the supremum over $k$ finally gives
\begin{align*}
    \sup_{t\in(0,T)\setminus \MB T_0}|v_\zeta(t)-v(t)| \leq \max_{k\in\inter{K}} e_{\zeta,k} \to 0 \quad\text{ as } \zeta\to +\infty. 
\end{align*}
Thus $v_\zeta\to v$ uniformly. 
\end{proof}

Having established the uniform convergence for the input neuron's mollified dynamics \eqref{eq:input_mollified}, we now pass to the hidden layers. The mechanism is entirely parallel: between spikes the trajectories follow the same linear RC flow (now driven by the presynaptic current $J_\ell$), and the only approximation occurs at resets through the same discharge map $D_\zeta$. Consequently, the time-change argument and the control of post-reset mismatches carry over verbatim, neuronwise and layerwise. For completeness and notational continuity, we state next the hidden-layer analogue; its proof is identical in structure to the input-layer case and is therefore omitted.

\begin{proposition}\label{prop:relaxation_hidden}
Let $T>0$ and for given $\ell\in\inter{L}$ and $p\in\inter{P}$ fix $\omega_{\ell,p}\in\RR$, $\tau_{\ell,p}>0$ and $\theta_{\ell,p}>0$. Let $\xi_{\ell,p}$ denote the solution of \eqref{eq:model_hidd} with spike times set $\MC T_{\ell,p}$. We assume all spikes are transversal threshold crossings, i.e. $\dot \xi_{\ell,p}(t^\ast)\neq 0$ for all $t^\ast\in\MC T_{\ell,p}$. Let $\{H_\zeta\}_{\zeta>0}$ be a family of smooth reset mollifiers $H_\zeta:\RR\to(0,1)$ defined as in \eqref{eq:mollifier}. For each $\zeta>0$ define the discharge map $D_\zeta(z;\theta)$ as in \eqref{eq:discharge} and let $\xi_{\zeta,\ell,p}$ be the solution to the reset–mollified hidden dynamics \eqref{eq:hidden_mollified}. Then 
\begin{align*}
    \lim_{\zeta\to +\infty} \sup_{t\in(0,T)}|\xi_{\zeta,\ell,p}(t)-\xi_{\ell,p}(t)|=0,    
\end{align*}
that is, $\xi_{\zeta,\ell,p} \to \xi_{\ell,p}$ uniformly on $(0,T)$.
\end{proposition}

The uniform convergence of trajectories established in Propositions \ref{prop:relaxation_input} and \ref{prop:relaxation_hidden} guarantees that, away from spike events, the mollified solutions shadow the hybrid ones arbitrarily closely as $\zeta\to +\infty$. However, for spiking systems the most relevant qualitative feature is not the continuous evolution between resets but the precise timing of threshold crossings. To pass from convergence of trajectories to convergence of spike trains, one must ensure that the times at which the mollified trajectories hit threshold converge to the spike times of the hybrid system. This is not automatic, since a uniform approximation of the state does not by itself prevent small displacements in the moments of threshold crossing. The key property that makes this possible is the transversality of crossings in the hybrid system: at each spike time, the potential crosses the threshold with nonzero slope. 

We emphasize that assuming transversality of threshold crossings is both natural and non-restrictive. In fact, transversality simply requires that the membrane potential crosses the threshold with nonzero velocity, which corresponds precisely to the operational definition of a spike in spiking neuron models. Non-transversal events, where the trajectory grazes the threshold tangentially, arise only under finely tuned parameter configurations and thus form a measure-zero set. Excluding them does not limit the generality of the analysis and ensures that spike times are well-defined and stable under perturbations. Under this assumption, standard perturbation arguments (implicit function theorem) ensure that the mollified trajectories admit corresponding threshold crossings whose timings depend continuously on $\zeta$. This observation motivates the following proposition.

\begin{proposition}\label{prop:spike_times_convergence}
Under the assumptions of Proposition \ref{prop:relaxation_input} (resp., \ref{prop:relaxation_hidden}), let $\{t_k\}$ denote the spike times of the hybrid system \eqref{eq:model_pre} (resp., \eqref{eq:model_hidd}). Then, for each $t^\ast_k$ there exists a spike time $t_{\zeta,k}$ of the $\zeta$-mollified system such that
\begin{align*}
    t_{\zeta,k} \to t_k \quad \text{ as } \zeta \to +\infty.    
\end{align*}
Moreover, the entire spike train of the mollified system converges to that of the hybrid system.
\end{proposition}

\begin{proof} For brevity, we treat only the input neuron; the argument for hidden neurons is identical and left to the reader.

Fix a spike time $t_k$ of the hybrid system \eqref{eq:model_pre}, i.e. $v(t_k^-)=\theta_v$ and $v(t_k^+)=0$. By assumption, this threshold crossing is transversal, meaning $\dot v(t_k^-) \neq 0$. 

From Proposition \ref{prop:relaxation_input}, we know that $v_\zeta \to v$ uniformly on $(0,T)$. In particular, in a neighborhood $[t_k-\delta, t_k+\delta]$, we have
\begin{align*}
    \sup_{t\in[t_k-\delta,t_k+\delta]} |v_\zeta(t)-v(t)| \to 0 \quad\text{ as }\zeta\to +\infty.    
\end{align*}
Thus, for large $\zeta$, the trajectory $v_\zeta$ remains uniformly close to $v$ in this neighborhood.

Now, because $\dot v(t_k^-)\neq 0$, there exists $\delta>0$ such that $v(t)<\theta_v$ for $t<t_k-\delta$ and $v(t)>\theta_v$ for $t\in(t_k-\delta,t_k)$ (or the opposite inequality, depending on the sign of the slope). Uniform convergence then implies that for large enough $\zeta$, $v_\zeta$ must also cross the threshold $\theta_v$ somewhere in $(t_k-\delta,t_k+\delta)$. Define this crossing as $t_{\zeta,k}$ and consider the smooth function
\begin{align*}
    h_\zeta(t) \coloneqq v_\zeta(t)-\theta_v.    
\end{align*}

We have $h_\zeta(t_{\zeta,k})=0$. Moreover, by construction, $h_\zeta\to h(t)=v(t)-\theta_v$ uniformly near $t_k$ and $\dot h(t_k^-)=\dot v(t_k^-)\neq 0$. For large $\zeta$, $\dot h_\zeta$ remains close to $\dot h$, so it does not vanish in a small neighborhood of $t_k$. Hence the implicit function theorem guarantees that the zero of $h_\zeta$ in that neighborhood is unique and depends continuously on $\zeta$.

Because $t_{\zeta,k}\in(t_k-\delta,t_k+\delta)$ by construction, this continuity of the zero w.r.t. uniform perturbations then implies that
\begin{align*}
    t_{\zeta,k} \to t_k \quad \text{ as } \zeta\to +\infty.    
\end{align*}

Finally, since the hybrid system has finitely many isolated spikes, the argument can be repeated for each $t_k$. Thus the entire spike train of the mollified system converges to that of the hybrid system, with each spike time converging individually.
\end{proof}

\section{Simulation experiments}\label{sec:experiments}

To complement the theoretical contributions developed in the preceding sections, we present here a set of simulation experiments designed to illustrate the performance of the proposed spiking neural network (SNN) model on a benchmark classification problem. Although intentionally simple, these experiments serve to validate our framework and to highlight the expressive and computational properties of the architecture in practice.

All experiments presented in this section, together with the corresponding code, are available on \texttt{GitHub} \cite{Github}.

\subsection{Dataset and architecture description}

We employ the \texttt{make\_moons} dataset from the \texttt{scikit-learn} library, a widely used synthetic dataset for binary classification. The data consist of two interleaving half-circles embedded in the plane, which form a nonlinear decision boundary that cannot be captured by linear classifiers. Points are generated by sampling from these two half-moon manifolds and assigning binary labels according to membership. 

The dataset can be enriched with Gaussian noise, thereby introducing variability and overlap between the classes. In this spirit, we considered two variants (see Figure \ref{fig:dataset}): one with 0\% noise, where the classes are perfectly separable, and one with 20\% noise, where a significant fraction of the samples are perturbed, yielding overlapping distributions. The first variant represents a clean classification setting, while the second one provides a more realistic and challenging scenario that allows us to assess robustness to noise. Moreover, in both scenarios, we have split the dataset among 70\% training data, 5\% validation data, and 25\% testa data.

 \begin{figure}[!ht]
     \centering     
     \includegraphics[scale=0.95]{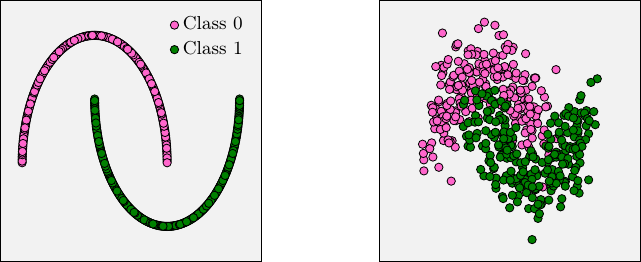}
     \caption{Training dataset with $0\%$ noise (left) and $20\%$ noise (right).}\label{fig:dataset}
 \end{figure}

Although the \texttt{make\_moons} dataset is simple compared to real-world data, it is well suited to our purposes. The main goal of this paper is not to propose large-scale benchmarks, but rather to investigate the mathematical underpinnings of SNN architectures, including their universal approximation properties and spike-trend dynamics. For this reason, a simple yet nonlinear dataset suffices to demonstrate the validity and potential of our framework.

As for the model's architecture, we employ a single hidden layer SNN ($L=1$) with $P=8$ neurons. Taking into account the input dimension $d=2$, this amounts at a total number of only $19$ trainable parameters. The architecture follows the dynamical modeling described in Sections \ref{sec:framework} and \ref{sec:training} of this paper, combining (possibly mollified) LIF dynamics in the input and hidden layers with a static readout for classification. The choice of a relatively small network is deliberate: given the limited complexity of the \texttt{make\_moons} dataset, adding depth or width would not necessarily improve performance. On the contrary, using a compact architecture demonstrates that our model achieves competitive results with minimal resources, highlighting the efficiency of the proposed design.

Finally, we have considered the following structural hyper-parameters for the LIF dynamics in our SNN architecture:
\begin{itemize}
    \item[] \textbf{Time horizon:} $T=60s$
    \item[] \textbf{Input layer:} $\tau_v=8$ and $\theta_v=0.8$
    \item[] \textbf{Hidden layers:} $\tau_{\ell,p}=6$ and $\theta_{\ell,p}=0.25$ for all $\ell\in\inter{L}$ and $p\in\inter{P}$
    \item[] \textbf{Output layer:} $\tau_u=10$ and $\theta_u=0.3$
    \item[] \textbf{Gaussian mollification:} $\mu=0.2$
\end{itemize}

As for the mollification hyperparameter $\zeta$, its value is not fixed but follows a smooth geometric schedule across training. Specifically, $\zeta$ is initialized at a low value $\zeta_0 = 3$ and gradually increased to $\zeta_1 = 10$ over the course of training epochs, according to the update rule 
\begin{align*}
    \zeta(e) = \zeta_0 \left(\frac{\zeta_1}{\zeta_0}\right)^{\frac{e}{E-1}},    
\end{align*}
where $e$ is the epoch index and $E$ the total number of epochs. This choice yields a progressive sharpening of the surrogate gradient while maintaining stability in the early stages. 

\subsection{Training and evaluation metrics}

Training is performed via surrogate-gradient optimization, following the regularized empirical risk minimization framework introduced in Section \ref{sec:training}. Since we are considering a binary classification problem, we have chosen the binary cross entropy as training loss function. To evaluate performance, we consider four standard classification metrics:
\begin{itemize}
    \item[] \textbf{Accuracy}, measuring the overall proportion of correct predictions;
    \item[] \textbf{Precision}, assessing the reliability of positive predictions;
    \item[] \textbf{Recall}, measuring the ability to correctly identify all positive samples;
    \item[] \textbf{F1}-\textbf{score}, the harmonic mean of precision and recall, providing a single measure that balances the trade-off between false positives and false negatives.
\end{itemize}

Simulations have been performed in \texttt{Python} on a Acer TravelMate P414-53 laptop with OS Ubuntu 24.04.3 LTS, 13th Gen Intel\textsuperscript{\textregistered} Core\textsuperscript{\texttrademark} i5-1335U × 12 processor and 32 GiB of RAM memory.

We report results on training, validation, and test splits, in order to assess generalization. On the dataset with 0\% noise, the model achieves accuracy, precision, and F1-score in the range 88\%–92\%, consistently across training, validation, and test sets. The recall metric is slightly lower, around 84\%, reflecting a mild tendency of the classifier to miss some positive samples (see Figure \ref{fig:metrics_clean} and Table \ref{tab:metrics}). Overall, these results confirm that the proposed SNN successfully captures the nonlinear decision boundary of the dataset while maintaining generalization.

 \begin{figure}[!ht]
     \centering     
     \includegraphics[scale=0.95]{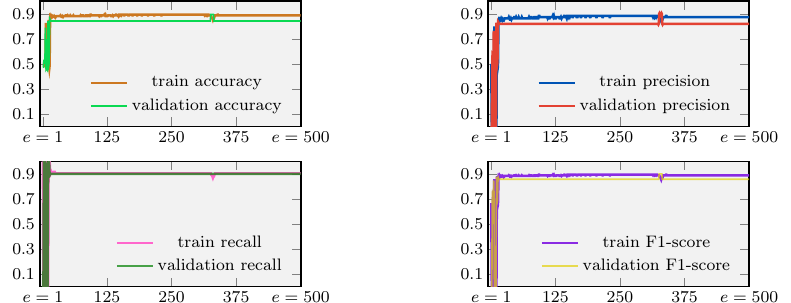}
     \caption{Training and validation metrics on the dataset with $0\%$ noise.}\label{fig:metrics_clean}
 \end{figure}

\begin{table}[!ht]
    \centering
    \begin{tabular}{|c|c|c|c|c|}
         \hline & \cellcolor{medioumgray}\textbf{Accuracy} & \cellcolor{medioumgray}\textbf{Precision} & \cellcolor{medioumgray}\textbf{Recall} & \cellcolor{medioumgray}\textbf{F1-score}
         \\
         \hline \cellcolor{medioumgray}\textbf{0\% noise} & 0.8889 & 0.9286 & 0.8387 & 0.8814
         \\
         \hline \cellcolor{medioumgray}\textbf{20\% noise} & 0.8413 & 0.8889 & 0.7742 & 0.8276
         \\
         \hline
    \end{tabular}
    \caption{Test metrics obtained by our SNN on the \texttt{make\_moons} dataset with 0\% and 20\% noise.}\label{tab:metrics}
\end{table}

On the dataset with 20\% noise, performance understandably decreases, with metrics lying in the range 75\%–88\% (see Figure \ref{fig:metrics_noise} and Table \ref{tab:metrics}). Among them, recall is again the lowest, reflecting the greater difficulty of recovering true positives in the presence of noisy, overlapping samples. Nonetheless, precision and accuracy remain close to 90\%, indicating that the model preserves a strong discriminative ability even under perturbations.

 \begin{figure}[!ht]
     \centering
     \includegraphics[scale=0.95]{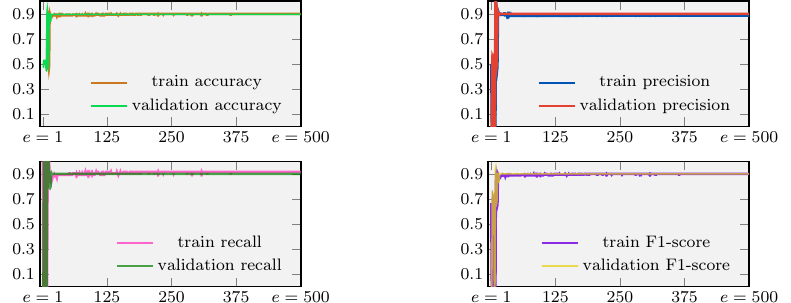}
     \caption{Training and validation metrics on the dataset with $20\%$ noise.}\label{fig:metrics_noise}
 \end{figure}

We shall also stress that, in both experimental scenarios we have considered, training and validation metrics are systematically higher than those obtained on the test set (around 90\% versus slightly lower values on test). This behavior is natural: since the test set contains previously unseen samples, including challenging points near the nonlinear boundary or affected by noise, slightly reduced performance is expected. Importantly, the gap remains modest, showing that the model does not suffer from significant overfitting.

These results support several key conclusions. First, they confirm that even a small SNN \textendash\, comprising only a single hidden layer with 8 neurons \textendash\, can efficiently solve nonlinear classification tasks, in line with the universal approximation principles established in Section \ref{sec:UA}. Second, the experiments show that the model remains robust to moderate levels of noise, maintaining good performance across metrics. Although recall drops both in the clean and noisy regimes, this effect is expected and does not compromise the overall effectiveness of the classifier. Indeed, the fact that accuracy and precision remain high underscores the stability and reliability of the spiking dynamics. Finally, the observation that training and validation metrics are consistently higher than test ones is not a drawback but rather an indication of good generalization: the model maintains strong performance even on unseen and more difficult data.

Overall, these experiments, while intentionally simple, validate the theoretical analysis developed throughout the paper. They demonstrate that the proposed SNN architecture is not only mathematically sound but also effective in practice, even under constrained architectures and noisy conditions.

\section{Conclusions and open problems}\label{sec:conclusions}

In this work we have developed a rigorous theoretical framework for Spiking Neural Networks (SNNs), focusing on two central aspects: their expressive power and the dynamical constraints that govern spike transmission. By proving a constructive universal approximation theorem for SNNs built from Leaky Integrate-and-Fire (LIF) neurons with threshold-reset dynamics, we established that narrow spiking architectures can approximate any continuous function on compact domains to arbitrary accuracy. Complementing this, our analysis of spike propagation across layers revealed a general non-increasing trend in spike counts, with precise criteria for the exceptional regimes \textendash\, such as overlapping presynaptic spikes and high gain-to-threshold ratios \textendash\, where spike counts may increase. Taken together, these results offer a principled mathematical foundation for understanding both the power and the limitations of SNNs, providing guarantees that support their role in classification, signal processing, and beyond.

At the same time, our findings open several new directions for further research. We conclude by highlighting a number of open problems which we believe are particularly promising:
\begin{itemize}
    \item[1.] \textbf{Depth}–\textbf{width tradeoffs in SNNs.} While our universality result is constructive, it does not quantify the minimal number of neurons per layer or the necessary depth for approximation within a given error tolerance. Establishing precise depth–width tradeoffs, analogous to those known for classical ANNs and ResNets \cite{alvarez2024interplay}, would provide sharper guidelines for designing efficient spiking architectures.
    \item[2.] \textbf{Quantitative approximation rates.} Our universal approximation Theorem \ref{thm:UA} guarantees density in the space of continuous functions but does not provide explicit approximation rates in terms of the number of neurons, spikes, or layers. Deriving such rates would enable comparisons between SNNs and traditional neural networks and clarify whether temporal coding confers efficiency advantages.
    \item[3.] \textbf{Training under hybrid dynamics.} The introduction of mollifiers in the LIF dynamics ensures well-posedness and differentiability, but the resulting surrogate models are only approximations of true spiking dynamics. A central challenge is to design principled training algorithms that directly exploit the hybrid nature of spike-reset systems while still yielding tractable gradients, perhaps by combining tools from optimal control and non-smooth analysis.
    \item[4.] \textbf{Generalization and capacity bounds.} Beyond approximation, little is known about the generalization properties of SNNs. Establishing statistical learning bounds \textendash\, such as Rademacher complexity \cite{bartlett2002rademacher} or VC-dimension \cite{vapnik1999overview} estimates tailored to spike-based models  \textendash\, remains an open question with immediate implications for their reliability in practical tasks.
    \item[5.] \textbf{Dynamics beyond LIF models.} Our analysis focused on LIF neurons, but biologically realistic neurons exhibit richer dynamics (e.g., adaptive thresholds, conductance-based models, or stochastic firing). Extending universal approximation results and dynamical analyses to these broader families would deepen the connection between theoretical SNNs and neuroscience.
\end{itemize}

In summary, this work provides a foundational step toward a mathematically rigorous theory of SNNs. Addressing the open problems listed above will be crucial to bridge the gap between their theoretical expressiveness and practical performance, ultimately advancing both neuromorphic computation and our understanding of biological neural systems.

\bibliographystyle{abbrv}
\bibliography{bibliography}

\end{document}